\pdfoutput=1
\documentclass[10pt]{article}
\usepackage[a4paper, margin=1in]{geometry}
\usepackage{changepage} 
\usepackage{mathtools}
\usepackage{amsmath, amssymb, amsthm}
\usepackage[utf8]{inputenc}
\usepackage[english]{babel}
\usepackage[numbers]{natbib}
\usepackage{url}
\usepackage[usenames]{color}
\usepackage{mathabx}
\usepackage[colorlinks=true]{hyperref}
\usepackage[capitalise]{cleveref}
\usepackage{tabularx}
\usepackage{breqn}
\usepackage[title]{appendix}
\usepackage{tikz}
\usepackage{xcolor}
\usepackage{comment}

\newtheoremstyle{mystyle}
  {}
  {}
  {\itshape}
  {}
  {\bfseries}
  {.}
  { }
  {\thmname{#1}\thmnumber{ #2}\thmnote{ (#3)}}
\theoremstyle{mystyle}
\newtheorem{theorem}{Theorem}

\newtheorem{lemma}{Lemma}

\newtheorem{cor}{Corollary}

\newtheorem{remark}{Remark}

\newcommand{\floor}[1]{\left\lfloor #1 \right\rfloor}

\newcommand{\HW}{\operatorname{HW}}

\title{Undecidability, Chaos and Universality in Arithmetic Terms}

\author{Gabriel Istrate\footnote{University of Bucharest, \tt{Gabriel.Istrate@gmail.com}}, Mihai Prunescu\footnote{Research Center for Logic, Optimization and Security (LOS) (Faculty of Mathematics and Computer Science, University of Bucharest, Academiei 14, Bucharest (RO-010014), Romania), e-mail address: {\tt mihai.prunescu@gmail.com}.}\footnote{ Simion Stoilow Institute of Mathematics of the Romanian Academy (Research unit 5, P. O. Box 1-764, Bucharest (RO-014700), Romania), e-mail address: {\tt mihai.prunescu@imar.ro}.}, Joseph M. Shunia\footnote{independent researcher, \tt{jshunia@gmail.com}}}

\date{}

\begin{document}

\maketitle

\begin{abstract} \noindent 
Arithmetic terms are finite fixed compositions of additions, subtractions, multiplications, divisions with remainder and exponentiations, containing variables interpreted as natural numbers. They build a well-defined notion of closed formula. It is known that every Kalmar elementary function can be expressed as an arithmetic term. In this paper, one studies the power of expression of the arithmetic terms. By interpreting Hilbert's Tenth Problem in arithmetic terms, it is shown that it is undecidable whether one-variable arithmetic term takes the value $0$, or whether two such terms take or not the same values. An algorithm constructs the arithmetic term representing an arbitrary function, which has been defined by a recurrence rule. This construction has various applications. Functions with chaotic behavior, like the Logistic Map, can be expressed as arithmetic terms. Finally, we construct a Turing complete arithmetic term and we express a Turing universal function by an arithmetic term. A somewhat unexpected application: there is a {\it wise} arithmetic term. It gets (the code of) a sentence, (the code of) a formalized theory and a bound $B$, and after performing a constant number of operations, it outputs (the code of) a proof of the sentence using the given theory if such a proof does exist and its length is less than $B$. Otherwise, it outputs $0$.  
\end{abstract}

\section{Introduction} \label{Sect.Introduction}

An arithmetic term is a finite expression built from variables ranging over $\mathbb N$ by composing a fixed stock of arithmetic operations, such as addition, truncated subtraction, multiplication, division with remainder, and exponentiation. In 1964, Rödding \cite{Rodding1964} (in German) showed that a finite quantity of elementary operations suffices to form a substitution basis for the class of Kalmar elementary functions, which marked the beginning of the search for substitution bases consisting of ``simple'' arithmetic operations. This search was considered concluded in 2002 (see Marchenkov \cite[Introduction]{marchenkov2007superposition}), when Mazzanti \cite{mazzanti2002plainbases} proved that \begin{equation*}  \langle x + y , \ x \dotdiv y , \ x y , \ \floor{x / y} , \ x^y \rangle = \mathcal{E}^3 \end{equation*} 
In \cite{prunescusaurasshuniaminimal} it is shown that 
\begin{equation*}  \langle x + y , \ x \bmod y , \ 2^x \rangle = \mathcal{E}^3 \end{equation*} 
and that this basis is minimal. However, we will work with arithmetic terms in the wider sense. In particular, we can use in a term both functions $\lfloor x/y \rfloor$ and $x \bmod y$ and the exponentiation in various forms like $2^x$, $7^x$ or $x^y$. Already represented functions will be used in short form, like macros. 

We collect below the main phenomena treated in this paper.
\begin{itemize}
    \item Hilbert's Tenth Problem can be interpreted in arithmetic terms: it is algorithmically undecidable whether a one-variable arithmetic term ever takes the value $0$, and the identity problem for one-variable arithmetic terms is also undecidable. This uses Matiyasevich's theorem \cite{matiyasevich1993hilbert}; see \cref{Sect.HTP}.
    \item Several useful elementary building blocks have explicit arithmetic-term representations, including binomial coefficients, $\gcd$, the $2$-adic valuation $\nu_2$, and the Hamming weight $\HW$. Relevant sources include \cite{mazzanti2002plainbases, marchenkov2007superposition, prunescushunia2024gcd, matiyasevich1993hilbert}; see \cref{Sect.StartingSteps}.
    \item Mazzanti's hypercube method \cite{mazzanti2002plainbases} gives arithmetic terms that count the bounded solutions of suitable exponential Diophantine equations. This solution-counting construction is one of the basic tools used throughout the paper; see \cref{Sect.Hypercube}.
    \item A general recurrent sequence whose recurrence rule is given by an arithmetic term, and for which a suitable arithmetic-term bound is available, has its $n$-th term represented by an arithmetic term. This is the main metatheorem, stated as \cref{metatheorem}.
    \item The method gives closed arithmetic-term descriptions for several dynamical examples: the rational Tent Map \cite{Bruin}, the rational Logistic Map \cite{May}, the odd Collatz sequence \cite{Collatz}, and finite-color approximations of Julia and Mandelbrot sets \cite{Julia}; see \cref{Sect.Physics}.
    \item There is a Turing complete arithmetic term. More precisely, a fixed term can simulate bounded computations of deterministic Turing machines \cite{Turing, AroraBarak}, and yields a universal term $\Upsilon$ for bounded computation; see \cref{Section.TCT}. The Turing complete term gets as arguments the code of a Turing Machine, the code of its initial configuration and a natural number $n$. The term outputs the configuration after the $n$-th steps after performing a fixed number of computations. 
    \item The universal construction can be expressed more readably with goto programs, which are Turing complete \cite{Schoning2008}. This gives a compiler-like translation from bounded algorithms to arithmetic terms, including examples for addition, multiplication, modulo, and exponentiation; see \cref{Section.Gototerm}.
    \item A ``wise'' arithmetic term can perform bounded proof search for recursively axiomatizable theories: given a formal statement and a bound, it returns a proof code if one exists below the bound and returns $0$ otherwise. The proof-checking step uses the standard polynomial-time verifiability of formal proofs, see for example \cite{immerman1999descriptive}; see \cref{theoremwiseterm}.
\end{itemize}

The paper is organized as follows. \Cref{Sect.HTP,Sect.StartingSteps,Sect.Hypercube} give the undecidability result and the technical arithmetic-term tools. \Cref{Sect.General} proves the recurrence-representation metatheorem, which is then applied to dynamical examples in \cref{Sect.Physics}. \Cref{Section.TCT,Section.Gototerm} construct universal terms for Turing machines and goto programs, and the final section applies this universality to bounded proof search.

\section{Hilbert's Tenth Problem in arithmetic terms}\label{Sect.HTP}

Using an easy interpretation of Hilbert's Tenth Problem, we show that it is algorithmically undecidable whether arithmetic terms take the value zero. A bijection  $C : \mathbb N \times \mathbb N \rightarrow \mathbb N$ is generally called a {\bf pairing function}. 

\begin{lemma}
    There are pairing functions $C(x, y)$such that the function itself and the inverse functions $L, R : \mathbb N \rightarrow \mathbb N$ with the property that
    $$\forall n\,\,C(L(n), R(n)) = n,$$
    are representable by arithmetic terms. 
\end{lemma}

\begin{proof}
    Cantor's Pairing Function
    $$C(x,y) = \left \lfloor \frac{(x+y)(x+y+1)}{2} \right \rfloor + x$$
    is an arithmetic term. The two inverse functions $L(n)$ and $R(n)$ are Kalmar elementary, so they can be represented by arithmetic terms by Mazzanti's Theorem. Such explicit terms are given in \cite{prunescusauras2025manyfunctions}. However, these terms are very complicated and so not appropriate for demonstrative use. We may alternatively consider the pairing function:
    $$C(x,y) = 2^x(2y+1)-1,$$
    with inverse functions:
    $$L(n) = \nu_2(n+1),$$
    $$R(n) = \left \lfloor \frac{\left ( \left \lfloor \frac{n+1}{2^{\nu_2(n+1)}} \right \rfloor -1 \right )}{2} \right \rfloor.$$
    The 2-adic valuation $\nu_2(n)$ is represented by a relatively simple arithmetic term, see \cite{mazzanti2002plainbases}. 
    $$\nu_2(n) = \floor{\frac{(\gcd(n, 2^n))^{n+1} \bmod (2^{n+1}-1)^2}{2^{n+1}-1}}.$$
\end{proof} 

Given any pairing function $C(x,y)$ with inverses $L(n)$ and $R(n)$ it is easy to see that:
$$\forall\, x_0, x_1, \dots , x_k \in \mathbb N \,\, \exists\, m \in \mathbb N\,\,\,\,R(m) = x_0 \wedge RL(m) = x_1 \wedge \dots \wedge RL^{k}(m) = x_k.$$

\begin{theorem}
    The following problem is algorithmically unsolvable: Given some univariate arithmetic term $t(m)$, decide whether there exists $m \in \mathbb N$ such that $t(m) = 0$. 
\end{theorem} 

\begin{proof}
    Consider an arbitrary polynomial Diophantine equation
    $$E(x_0, x_1, \dots, x_k) = 0,$$
    where $E(x_0, x_1, \dots, x_k) \in \mathbb Z[x_0, x_1, \dots, x_k]$. To ensure that the polynomial from the left-hand side takes always non-negative values, we square the left-hand side, so we consider the equation:
    $$E(x_0, x_1, \dots, x_k)^2 = 0.$$
    This equation can be rewritten
    $$E^2_+(x_0, x_1, \dots, x_k) - E^2_-(x_0, x_1, \dots, x_k) = 0,$$
    where both polynomials $E^2_+, E^2_-$ have positive coefficients only. As we know that $$\forall \,\vec x \in \mathbb N^{k+1}\,\, E^2_+(\vec x) \geq E^2_-(\vec x),$$ we may replace the algebraic subtraction by the truncated difference:
    $$E^2_+(x_0, x_1, \dots, x_k) \dotdiv E^2_-(x_0, x_1, \dots, x_k) = 0.$$
    Now we introduce a new variable $m$ ranging over the natural numbers and we consider the arithmetic term
    $$t(m) := E^2_+(R(m), RL(m), \dots, RL^k(m)) \dotdiv E^2_-(R(m), RL(m), \dots, RL^k(m)).$$
    It is clear that the equation $E(x_0, \dots, x_k) = 0$ has  solutions in $\mathbb N^{k+1}$ if and only if the equation $t(m) = 0$ has solutions in $\mathbb N$. As the former question is algorithmically unsolvable by Matiyasevich, so is also the later. 
\end{proof} 

The identity problem between arithmetic terms is the question whether for arithmetic terms $t_1(n)$ and $t_2(n)$ the following proposition is true:
$$\forall \,n \in \mathbb N\,\,\,\, t_1(n) = t_2(n).$$

\begin{cor}
    The identity problem between arithmetic terms is algorithmically undecidable.
\end{cor}

\begin{proof}
    Consider again Hilbert's Tenth Problem in the form:
    $$\exists \,n \in \mathbb N\,\,\,\,t(n) = 0.$$
    This question has a negative answer if and only if:
    $$\forall \,n \in \mathbb N\,\,\,\, 1 \dotdiv t(n) = 0.$$
\end{proof} 

\section{Starting steps in arithmetic terms}\label{Sect.StartingSteps}

Consider the following functions: the binomial coefficient $\binom{a}{b}$, the greatest common divisor $\gcd(a,b)$, the $2$-adic valuation $\nu_2(n)$, the number of ones in binary representation $\HW(n)$, called also the Hamming Weight. They are represented by arithmetic terms, see \cite{mazzanti2002plainbases, marchenkov2007superposition, prunescushunia2024gcd}:
\begin{align*}
\binom{a}{b} &= \floor{\frac{(2^a+1)^a}{2^{ab}}} \bmod 2^a , \\
\gcd(a,b) &= \left ( \floor{\frac{5^{ ab(ab + a + b)}}{(5^{a^2 b} - 1)(5^{ab^2}-1)}} \bmod 5^{ab} \right ) - 1, \\
\nu_2(n) &= \floor{\frac{(\gcd(n, 2^n))^{n+1} \bmod (2^{n+1}-1)^2}{2^{n+1}-1}} , \\
\HW(n) &= \nu_2\left(\binom{2n}{n}\right) .
\end{align*}  

Mazzanti proved another $\gcd$-formula in \cite{mazzanti2002plainbases}. 

The exponent $5$ given in the present formula can be replaced by $2$, but the the formula does not work anymore in the case $(a,b) = (1,1)$. The exponent base $5$ can be replaced by the exponent basis $2$ also using Marchenkov's formula:
$$a^b = 2^{(ab + a + 1)b} \bmod (2^{ab + a + 1} - a),$$
see \cite{marchenkov1980superposition}, and the same can be done for the general exponentiation in the closed form for the binomial coefficient. Many occurrences of $a^b$, with various expressions for $a$ and $b$, can be handled this way. 

The term for $HW(n)$ displayed above is a consequence of Kummer's Theorem, see Matiyasevich \cite{matiyasevich1993hilbert}.

Consider the generalized geometric series:
\begin{align}
S_r (q, t) = \sum_{j=0}^{t} j^r q^j .
\end{align} 
Matiyasevich shows that for every $r \in \mathbb N$,  $S_r(q,t)$ is represented by an arithmetic term $G_r(q,t)$, see \cite{matiyasevich1993hilbert}, in Appendix. For example, 
$$G_0(q,t) = \frac{q^{t+1}-1}{(q-1)},$$
$$G_1(q,t) = \frac{tq^{t+2} - (t+1)q^{t+1} +1}{(q-1)^2} ,$$
$$G_2(q,t) = \dfrac{t^2 q^{t+3} - (2 t^2 + 2t - 1)q^{t+2} + (t+1)^2 q^{t+1} - q^2 - q}{( q - 1 )^3},$$
and so on. These terms are called generalized geometric series. The family $\left (G_r(q,t) \right )$ is not uniform in $r$, and their sizes are unbounded. Nevertheless, there exists also an arithmetic term $G(r, q, t)$ such that for all $r, q, t \in \mathbb N$,
$$G(r, q, t) = G_r(q,t).$$
This fact can be proved ad hoc, but it will follow also from the results of the present article. Nevertheless, such a term is very complicated, and so is not useful in experiments. On the other hand, it is proved in \cite{prunescusauras2025manyfunctions} that every Kalmar elementary function can be expressed using the functions $G_r(q,t)$ with $r \in \{0, 1, 2\}$ and the function $HW(n)$.

\section{The solution-counting term}\label{Sect.Hypercube}

In this section we explain Mazzanti's construction of a term that counts the natural number solutions of an exponential Diophantine equation inside a given hypercube of edge-length $t-1$. This method is sometimes called {\bf the hypercube method}, see \cite{mazzanti2002plainbases}. 

An exponential expression
$$E(x_1, x_2, \dots, x_n) $$ 
is called {\bf simple} if it is a sum of simple monomials. A simple monomial has the form:
$$a a_1^{x_1}\dots a_n^{x_n} x_1^{b_1} \dots x_n^{b_n}$$
where $x_1, \dots, x_n$ are unknowns, $a, a_1, \dots, a_n, b_1, \dots, b_n$ are integers,  $a_1, \dots, a_n \geq 1$  and $b_1, \dots, b_n \geq 0$. We are interested in counting the integer solutions $(x_1, \dots, x_n) \in [0,t-1]^n$ of the equation
$$E(x_1, x_2, \dots, x_n) = 0$$
for some $t \in \mathbb N$. We will see that the number of solutions
$$F(t, \vec a) = | \{(x_1, \dots, x_n) \in [0,t-1]^n \, |\, E(x_1, \dots, x_n) = 0 \}|$$
is expressible by an arithmetic term in $t$ and $\vec a$. The parameter vector $\vec a$ consists of the monomial coefficients $a$ and of the exponentiation bases $a_i$ of the various simple monomials occurring in $E(x_1, \dots, x_n)$. All these parameters occur explicitly in the term $F(t, \vec a)$. The shape of the term depends on the polynomial exponents $b_1, \dots, b_n$ from various monomials, but these integers do not occur explicitly in the arithmetic term $F(t, \vec a)$. 

Without lost of generality, we may suppose that $E(\vec x) \geq 0$ for all $\vec x \in {0, \dots, t-1}^n$, because we can replace the exponential expression $E(x_1, \dots, x_n)$ with its square. In practical applications, the expression $E$ often arises as a sum of squares. We also choose a $w \in \mathbb N$ such that $E(\vec x) \leq 2^w$ for all $\vec x \in \{0, \dots, t-1\}^n$. 

Given two integers $a$ and $w$ such that $0 \leq a < 2^w$, let 
$$\delta(a,w) = (2^w - 1)(2^w - a + 1) =  2^{2 w} - 2^w a + a - 1 .$$
Recall that $\HW(x)$ means the number of digits equal one in the base-two representation of $x$. The function $\HW(x)$ is expressed by an arithmetic term. It is easy to see that:
$$\HW(\delta(a, w)) = \begin{cases}
    2w, & a=0, \\
    w, & a \neq 0.
\end{cases}
$$ 
Let $ v $ denote the function that maps each point $ \vec{x} \in \{ 0 , \dots , t  - 1 \}^n $ into the number $ x_1 + x_2 t  + \dots + x_n t^{n - 1} $.

Observe that $ v $ enumerates the points of $ \{ 0 , \dots , t - 1 \}^n $ from $ 0 $ to $ t^n - 1 $.

Let $$ M ( t ) = \sum_{\vec{x} \in \{ 0 , \dots , t - 1 \}^n} 2^{2 w v ( \vec{x} )} \delta ( E ( \vec{x} ) , w) , $$ which is well-defined because $ 0 \leq E ( \vec{x} ) < 2^{w} $ for every $ \vec{x} \in \{ 0 , \dots , t - 1 \}^k $, and let $ d ( t) $ denote the cardinality of the set $ \{ \vec{x} \in \{ 0 , \dots , t - 1 \}^n : E ( \vec{x} ) = 0 \} $.

We observe that the base-two representation of the number $M(t)$ is the concatenation of the base-two representations of the numbers $\delta(E(x), w)$. It follows that:
$$d(t) = \left \lfloor \frac{ \HW(M(t))}{w} \right \rfloor -t^n,$$
and this would be an arithmetic term if $M(t)$ would be represented by an arithmetic term. To this end, we recall that the arithmetic-geometric sums $G_u(v, t)$ are arithmetic terms, and we observe that:

 $$ \sum_{\vec{x} \in \{ 0 , \dots , t - 1 \}^n} x_1^{b_1} a_1^{x_1} \dots x_n^{b_n} a_n^{x_n} = G_{b_1} ( a_1 , t - 1 ) \dots G_{b_n} ( a_n , t - 1 ) . $$ 

As the in the definition of $\delta(a,w)$, the quantity $a$ participates as an affine variable, the sum defining $M(t)$ can be reformulated such that we can apply the aforementioned identity separately for every simple exponential monomial. So we find out that the contribution of the free (constant) monomial $\varepsilon$ is:
$$ \mathcal{C} ( \varepsilon)  = ( 2^{w } - \varepsilon  + 1 ) \left \lfloor \frac {( 2^{2 w  {t }^n } - 1 ) }{ ( 2^{w } + 1 )} \right \rfloor , $$ 
while the contribution of the simple exponential  monomial $m = a \, a_1^{x_1}\dots a_n^{x_n} x_1^{b_1} \dots x_n^{b_n}$ is
$$  \mathcal{A} ( m  ) = - ( 2^{w} - 1 ) \, a \, G_{b_1} ( 2^{ 2 w } a_1, t  - 1 ) \dots G_{b_n} ( 2^{ 2 w  {t }^{n - 1}} a_n, t  -1 ) . $$ 

This construction will be applied in the following way. Suppose that we know that a function of interest $f(\vec m)$ is equal with the number of solutions of a parametric exponential Diophantine simple non-negative equation $E(\vec m, \vec x) = 0$ in a hypercube $\{0, \dots, t(\vec m) - 1\}^n$. Suppose that the dependency $t(\vec m)$ can be expressed as an arithmetic term and that the way in which every simple exponential monomial $m = a \, a_1^{x_1}\dots a_n^{x_n} x_1^{b_1}$ depends on $\vec m$  is that $a = a(\vec m)$, $a_1 = a_1(\vec m)$, $\dots$, $a_n = a_n(\vec m)$ depend on $\vec m$ as arithmetic terms, while the polynomial exponents $b_1$, $\dots$, $b_m$ are always constants. In this case one can always find an arithmetic term $w(\vec m)$ such that for all $\vec x \in \{0, \dots, t(\vec m) - 1\}$, $0 \leq E(\vec m, \vec x) < 2^{w(\vec m)}$. It follows that the function of interest $f(\vec m)$ will be expressed by an arithmetic term:
$$f(\vec m) = d(t(\vec m), w(\vec m)) = \left \lfloor \frac{ \HW(M(t(\vec m), w(\vec m))}{w(\vec m)} \right \rfloor -t(\vec m)^n,$$

\section{Represent general recurrent sequences}\label{Sect.General}

In this section we exemplify a general method to find the general closed form expressing the $n$-th term of a recurrent sequence. 

\begin{theorem}\label{metatheorem}
The following data is given:
\begin{enumerate}
\item A general recurrence rule of the shape:
$$a(n+k) = F(n, a(n), \dots, a(n+k-1)),$$
where $F(n, x_0, \dots, x_{k-1})$ is a known arithmetic term.
\item Initial values $a(0), \dots, a(k-1) \in \mathbb N$.
\item A function $A(n, x_0, \dots, x_{k-1})$ such that for all $n \in \mathbb N$,  and for all $j$ with $0 \leq j \leq n $, one has:
$$a(j) < A(j, a(0), \dots, a(k-1)).$$ 
\end{enumerate}
Then one can always construct a term $T(n, x_0, \dots, x_{k-1})$, in the language $L_A = \{+, \dot{-}, \cdot, \lfloor x/y \rfloor, 2^x\} \cup \{A(n, x_0, \dots, x_{k-1})\}$,  such that for all $n \in \mathbb N$,
$$a(n) = T(n, a(0), \dots, a(k-1)).$$
\end{theorem}

The function $A(n, x_0, \dots, x_{k-1})$ is not always representable by an arithmetic term. This can be seen easily for the sequence:
$$a(0) = 1\,\,\,\, \wedge\,\,\,\, \forall\,n \geq 1 \,\,\,\,a(n+1) = 2^{a(n)}.$$
It follows immediately:
\begin{cor}
    If the sequence $(a(n))$ is recurrent of order $k$,  the recurrence is  an arithmetic term $$ F(n, a(n), \dots, a(n+k-1)),$$ and the sequence $(a(n))$ is bounded by an arithmetic term $$A(n, a(0), \dots, a(k-1)),$$ then there exists an arithmetic term $$T(j, a(0), \dots, a(k-1))$$ expressing the member $a(j)$. 
\end{cor}

We will effectively apply the following Lemma:

\begin{lemma}\label{lemma:fromtermtoeq}
    Let $B(\vec x)$ be any arithmetic term. Then the relation:
    $$B(\vec x) = y$$
    is equivalent with a simple exponential Diophantine relation:
    $$\exists \, \vec z\,\, E_B(\vec x, y, \vec \lambda) = 0.$$
    Here:
\begin{enumerate}
    \item $E_B(\vec x, y, \vec \lambda)$ is a sum of squares.
    \item If $B(\vec x) = y$ then there is a unique $\vec \lambda$ such that $E(\vec x, y, \vec \lambda) = 0$.
    \item There are arithmetic terms $L_i(\vec x, y)$ such that  $\lambda_i < L_i(\vec x, y)$.
\end{enumerate}
\end{lemma} 

\begin{proof}
   The proof is done by induction on the inductive definition of arithmetic terms. If the term $B(\vec x) = x_i$ is a projection, then the equation will be:
   $$(x_i - y)^2 =0. $$
   Suppose that two relations $B_i(\vec x) = y_i$, where $i=1$ or $i=2$, have been already proven equivalent with exponential Diophantine equations:
   $$E_{B_i}(\vec x, y_i, \vec \lambda_i) = 0$$
   which fulfills all properties given in the statement. 

   In [Prunescu, Sauras, Shunia, Basis paper] it was shown that all Kalmar elementary functions can be represented using only the functions $2^x$, $x+y$ and $x \bmod y$, so it suffices to analyze these steps. 

   {\it Case 1:} The relation has the shape:
   $$2^{B_1(\vec x)} = y.$$ 
   We write the equation:
   $$(2^{y_1} - y)^2 + E_{B_1}(\vec x, y_1, \vec \lambda_1) = 0,$$
   which we denote:
   $$E_{2^{B_1}}(\vec x, y, \vec \lambda) = 0$$
   and the tuple $\vec \lambda = (\vec \lambda_1, y_1)$. We observe that $y_1 < B_1(\vec x) + 1 = L_1(\vec x, y)$. 

   {\it Case 2:} The relation has the shape:
   $$B_1(\vec x) + B_2(\vec x) = y.$$
   We write the equation:
   $$(y - y_1 - y_2)^2 + E_{B_1}(\vec x, y_1, \vec \lambda_1) + E_{B_2}(\vec x, y_2, \vec \lambda_2) = 0, $$
   which we denote:
   $$E_{B_1 + B_2}(\vec x, y, \vec \lambda) = 0$$
   and the tuple $\vec \lambda = (\vec \lambda_1, y_1, \vec \lambda_2, y_2)$. We observe that $y_i < B_i(\vec x) + 1 = L_i(\vec x, y)$.

   {\it Case 3:} The relation has the shape:
   $$B_1(\vec x) \bmod B_2(\vec x) = y.$$
   We write the equation:
   $$[(y_1 - q y_2 - y)^2 + (y_2 - y_3 - y - 1)^2] [y_2^2 + y_3^2 + (y-y_1)^2 + q^2] + E_{B_1}(\vec x, y_1, \vec \lambda_1) + E_{B_2}(\vec x, y_2, \vec \lambda_2) = 0,$$
    which we denote:
   $$E_{B_1 \bmod B_2}(\vec x, y, \vec \lambda) = 0$$
   and the tuple $\vec z = (\vec z_1, y_1, \vec z_2, y_2, y_3, q) $.  The first factor can be satisfied only if $B_2(\vec x) \neq 0$, and says that $y$ is indeed the remainder of the division of $B_1(\vec x)$ by $B_2(\vec x)$. The second factor can be satisfied only if $B_2(\vec x) = 0$. It forces $y_1 \bmod 0 = y_1$ according to the generally accepted convention. But any other convention can be expressed by an exponential Diophantine equation as well.  

   We observe again that $y_1 < L_1(\vec x, y) $, $y_2 < L_2(\vec x, y)$, $y_3 < L_2(\vec x, y)$ and $q < L_1(\vec x, y)$. 
\end{proof} 

We also make use of the positional coding. The sequence of natural numbers 
$$a_0, a_1, \dots , a_n$$
satisfying $a_k < A$ for all $k \in \{0, 1, \dots, n\}$, is uniquely encoded by
$$x = a_0 + a_1 A + \dots + a_n A^n.$$
Let $\textrm{Code}(x, A, n)$ mean that $x$ encodes a sequence bounded by $A$, of length $n+1$. Then:
$$\textrm{Code}(x, A, n) \longleftrightarrow x < A^{n+1}.$$
If we denote by $\pi(x, A, j, n)$ the element $a_k$ of the sequence, then:
$$a = \pi(x, A, j, n) \longleftrightarrow x = \lambda_1 + a A^j + \lambda_2 A^{j+1} \wedge \,\, \lambda_1 < A^j \wedge a < A \wedge j \leq n. $$ 

\begin{lemma}\label{lemma.coding}
    The relations $\textrm{Code}(x, A, n)$ and $a = \pi(x, A, j, n)$ have simple exponential Diophantine definitions 
    $$E_{\textrm{Code}}(x, A, n, \lambda) = 0,$$
    $$E_\pi(x, A, j, n, a, \vec \lambda) = 0.$$
    Here:
    \begin{enumerate}
        \item Both expressions $E_{\textrm{Code}}$ and $E_\pi$ are sums of squares.
        \item If  $\textrm{Code}(x, A, n)$, respectively $a = \pi(x, A, j, n)$ are true, then there exist unique tuples $\vec \lambda$ satisfying the equations.
        \item  There are arithmetic terms $L_i(x, A, j, n, a)$ such that $\lambda_i < L_i(x, A, j, n, a)$.
    \end{enumerate}
\end{lemma}

\begin{proof} All these conditions are exponential Diophantine. Indeed, the first can be written:
$$\exists \, \lambda\,\,\,\,(A^{n+1} - x - \lambda - 1)^2 = 0,$$
$$E_{\textrm{Code}}(x, A, n, a, \lambda) = 0.$$
Here $\lambda < A^{n+1} = L(A, n)$. For the second relation:
$$ \exists \, \vec \lambda\,\,\,\,( x - \lambda_1 - a A^j - \lambda_2 A^{j+1})^2 + (A^j - \lambda_1 - \lambda_3 - 1)^2 + (A - \lambda_4 -1)^2 + (n-j-\lambda_5)^2 = 0,   $$
$$E_\pi(x, A, j, n, a, \vec \lambda) = 0.$$
We observe that $\lambda_1 < A^j = L_1(A, j)$, $\lambda_2 < x+1 = L_2(x)$, $\lambda_3 < A^j = L_3(A, j) $, $\lambda_4 < A = L_4(A)$, 
$\lambda_5 < n+1 = L_5(n)$. 

\begin{proof} (of \cref{metatheorem}) The natural number $k$, meaning the order of the recurrence relation, will be treated as a constant. Let $n \geq k $ be some natural number. Consider the following exponential Diophantine relation:
$$E_{Code}(x, A(n, y_0, \dots, y_{k-1}), n, \lambda_0) 
+ E_\pi(x, A(n, y_0, \dots, y_{k-1}), 0, n, y_0, \vec \lambda_1) + $$
$$ + E_\pi(x, A(n, y_0, \dots, y_{k-1}), 1, n, y_1, \vec \lambda_2) + \dots + E_\pi(x, A(n, y_0, \dots, y_{k-1}), k-1, n, y_{k-1}, \vec \lambda_k) +$$ 
$$ + E_\pi(x, A(n, y_0, \dots, y_{k-1}), j, n, z_0, \vec \lambda_{k+1}) + \dots + E_\pi(x, A(n, y_0, \dots, y_{k-1}), j + k-1, n, z_{k-1}, \vec \lambda_{2k}) +$$
$$ + E_\pi(x, A(n, y_0, \dots, y_{k-1}), j+k, n, z_k, \vec \lambda_{2k+1}) + E_F(j+k, z_0, \dots, z_{k-1}, z_k, \vec \lambda_{2k+2}) =0.$$
This relation says that:
\begin{enumerate}
    \item The number $x$ is the positional code of a finite sequence $x_0, \dots, x_{n}$ in the base $A(n, y_0, \dots, y_{k-1})$.
    \item The first $k$ terms are identic with $y_0, \dots, y_{k-1}$. 
    \item For some place $k \leq j+k \leq n$, we denote the elements 
    $x_j, x_{j+1}, \dots, x_{j+k}$ with $z_0, \dots, z_k$.
    \item The recurrence rule:
    $$z_{j+k} = F(j+k, z_0, \dots, z_{j+k-1})$$
    is fulfilled. 
\end{enumerate}
We denote this equation as:
$$E(n, x, \vec y)(j, \vec z, \vec \lambda) = 0. $$
In this expression, we consider the letters $(n, x, \vec y )$
parameters, where $\vec y = (y_0, \dots, y_{k-1})$. The letters $(j, \vec z, \vec \lambda)$ are considered unknowns. Here $\vec z = (z_0, \dots, z_{k})$, while the tuple $\vec \lambda$ is the concatenation of the tuples $\lambda_0$, $\vec \lambda_1$, $\dots$, $\vec \lambda_{2k+2}$.

We observe the following property of this exponential Diophantine equation: for given $(n,x,\vec y)$, and for every $0 \leq j \leq n-k$, there is at most one tuple $(\vec z, \vec \lambda)$ such that $E(n, x, \vec y)(j, \vec z, \vec \lambda) = 0 $. So, as $0 \leq j \leq n-k$, this equation has for given values of the parameters $(n,x,\vec y)$ at most $n-k + 1$ solutions.

We formulate and prove two claims that help us to apply here the hypercube method shown in \cref{Sect.Hypercube}. 

{\it Claim 1:} {\it There is an arithmetic term $t(n,x, \vec y)$ such that}
$$\forall \,(n,x, \vec y) \in \mathbb N^{k+2}\,\,\forall \, (j, \vec z, \vec \lambda) \in \mathbb N^m $$ $$ E(n, x, \vec y)(j, \vec z, \vec \lambda) = 0 \rightarrow (j, \vec z, \vec \lambda) \in [0, t(n, x, \vec y)]^m $$

{\it Proof of Claim 1:} We have seen that $j \leq n-k$. Also, every element of the tuple $\vec z$ is $\leq x$ which is itself $< A(n, \vec y)$. For the components $\lambda_i$ of the tuple $\vec \lambda$, we observe that there are arithmetic terms $L_i(n, x, \vec y)$ such that $\lambda_i < L_i(n,x,\vec y)$. Finally, we recall that the functions:
$$|a-b| := (a \dot{-} b) + (b \dot{-} a),$$
$$\max(a,b) := \left \lfloor \frac{(a+b) \dot{-} |a-b| }{2}  \right \rfloor,$$
are arithmetic terms, and also the functions:
$$\max(a, b, c) := \max(\max(a,b),c),$$
$$\max(a, b, c, d) := \max(\max(a,b,c), d),$$
and so on, are all arithmetic terms. \qed 

{\it Claim 2:} {\it There is an arithmetic term $w(n,x,\vec y)$ such that}
$$\forall \,(n, x, \vec y) \in \mathbb N^{k+2} \,\, \forall \,(j, \vec z, \vec \lambda)  \in [0, t(n, x, \vec y)]^m   $$
$$0 \leq E(n, x, \vec y)(j, \vec z, \vec \lambda) < 2^{w(n,x,\vec y)}.$$

{\it Proof of Claim 2:} The standard argument for the existence of such an arithmetic term, is the following. One replaces in the exponential expression $E(n, x, \vec y)(j, \vec z, \vec \lambda)$ every minus sign with a plus sign, and one replaces every of the $m$ unknowns $(j, \vec z, \vec \lambda)$ with $t(n, x, \vec y)$. One can define $w(n,x, \vec y)$ to be the resulting arithmetic term. \qed 

{\it Claim 3:} {\it There is an arithmetic term $G(n,x,\vec y)$ such that} 
$$\forall \,(n, x, \vec y) \in \mathbb N^{k+2} \,\,\,\, G(n, x, \vec y ) = | \{(j, \vec z, \vec \lambda)\,:\, E(n, x, \vec y)(j, \vec z, \vec \lambda) =0 \}|. $$

{\it Proof of Claim 3:} This follows from {\it Claim 1}, {\it Claim 2} and \cref{Sect.Hypercube}. \qed 

We have already seen that for all $n \geq k$ and all $(x, \vec y) \in \mathbb N^{k+1}$,
$$G(n, x, \vec y) \leq n-k+1.$$
Now we rewrite the term $G(n,x, \vec y)$ as $G(n, \vec y)(x)$. We consider $n \geq k$ and $\vec y$ to be parameters, while $x$ is the only one unknown. 

{\it Claim 4:} {\it For given values $n \geq k$ and $\vec y \in \mathbb N^k$, the equation
$$G(n, \vec y)(x) = n-k+1$$
has a unique solution $x = C \in \mathbb N$. This solution is the positional code of the finite sequence $a(0), a(1),$ $ \dots,$ $ a(n)$ that satisfies the recurrence rule:
$$a(j+k) = F(j, a(j), \dots, a(j+k-1)),$$
with initial values $a(0) = y_0$, $\dots$, $a(k-1)=y_{k-1}$, encoded with coding base $A(n, y_0, \dots, y_{k-1})$.
} 

{\it Proof of Claim 4:} The value of $G(n, \vec y)(x)$ is the number of $j$ with $k \leq j + k\leq n$ for which $ a(j+k) := \pi(x, A, j+k, n) $ satisfies the recurrence rule $a(j+k) = F(j, a(j), \dots, a(j+k-1))$, where all $a(l) = \pi(x, A, l, n)$ for $j \leq l \leq j+k-1$ and $A = A(n, \vec y)$. If this number equals its maximum $n-k+1$, all elements $\pi(x, A, j+k, n)$ satisfy the recurrence rule. But the first elements of the sequence respect the initial conditions. \qed 

Now let $E_1(n, \vec y, x, M ,\vec \mu) = 0$ be an exponential Diophantine equation which, for every $(n, \vec y, M)$ satisfies:
$$\exists \, \vec \mu \,\,\,\, E_1(n, \vec y, x, M, \vec \mu) = 0 \longleftrightarrow G(n, \vec y)(x) = M.$$
This equation exists according to \cref{lemma:fromtermtoeq}. We write this equation in the form:
$$E_1(n, \vec y, M)( x,  \vec \mu) = 0.$$
For the value $M = n-k+1$, we already know that the equation:
$$E_1(n, \vec y, n-k+1)( x,  \vec \mu) = 0$$
has exactly one solution $(x, \vec \mu)$ with $x = C$. 
\end{proof} 
We introduce two new unknowns $\omega_1$ and $\omega_2$ and we consider the equation:
$$E_1(n, \vec y, n-k+1)( \omega_1 + \omega_2 + 1,  \vec \mu) = 0,$$
that we write as:
$$E_2(n, \vec y)(\omega_1, \omega_2, \vec \mu) = 0,$$
where we recall that $k$ is a constant. 

{\it Claim 5:} {\it The number of solutions $(\omega_1, \omega_2, \vec \mu)$ is equal with the unique code-number $x = C$. }

{\it Proof of Claim 5:} As the unique solution of the equation $E_1(n, \vec y, n-k+1)( x,  \vec \mu) = 0$ is a tuple $(x, \vec \mu) = (C, \vec \mu)$, the number of solutions $(\omega_1, \omega_2, \vec \mu)$ of $E_2$ equals the number of ways in which one can write $C$ in the form $\omega_1 + \omega_2 + 1$. But this number is equal $C$. \qed

Now, by applying {\it Claim 1} and {\it Claim 2} for the exponential Diophantine equation $E_2$ we find corresponding terms $t_2(n, \vec y)$ and $w_2(n, \vec y)$. We apply again the construction from \cref{Sect.Hypercube}, and we find an arithmetic term $H(n, \vec y)$ such that:
$$H(n, \vec y) = C.$$
The sequence-term $a(n)$ is the projection:
$$a(n) = \pi(C, A, n, n) = \pi(H(n, \vec y), A(n, \vec y), n, n),$$
where in general
$$\pi(C, A, n, n) = \left \lfloor \frac{C}{A^{n}} \right \rfloor.$$
Altogether, 
$$a(n) = \left \lfloor \frac{H(n, \vec y)}{A(n, \vec y)^{n}} \right \rfloor $$
and this is the arithmetic term representation we were looking for. 
\end{proof}

\section{Some examples from Physics}\label{Sect.Physics}

The method presented in \cref{Sect.General} will be applied in the next section to construct a Turing complete arithmetic term. In this section we present some easier examples, based on recurrent sequences relevant in Physics. Moreover, we show here how the method can be used to find closed forms for sequences whose elements are not necessarily natural numbers, but can be general integers, rational numbers, elements of $\mathbb Q[i]$ or elements of even more complicated computable rings.

\subsection{The Tent Map}\label{Subsect.TentMap} 

The Tent Map, see for example \cite{Bruin}, is defined by the recurrence $x(n+1) = f_\mu(x(n))$ where

$$f_\mu(x) = \mu \min( x, 1-x), $$

with given $x(0) \in [0,1]$ and $\mu \in [0, 2]$. We see that all elements of the sequence belong to $[0,1]$. Indeed,  if $x \in [0,1]$ then 
$$f_\mu(x) = \mu \min( x, 1-x)\leq 2 \left(\frac{1}{2}\right)= 1.$$

In order to represent the Tent Map by arithmetic terms, we consider the case with $x(0), r \in \mathbb Q$. 

Let $x(0) = a/b$ and $r = p/q $.

If we write $x(n) = a(n) / b(n)$ in natural numbers $a(n)$ and $b(n)$.

The recurrence rule becomes: 

$$\frac{a(n+1)}{b(n+1)}    = \frac{p}{q} \cdot \min \left (\frac{a(n)}{b(n)}, \frac{b(n) - a(n)}{b(n)} \right ).$$

We identify numerators and denominators and we obtain the 
crossed recurrence rules:
$$a(n+1) = p \min( a(n) ,  b(n) - a(n) ),$$
$$b(n+1) = q b(n).$$
for two sequences of natural numbers, with initial conditions $a(0) = a$ and $b(0) = b$. 

For the sequence $b(n)$ we get the following arithmetic term:
$$b(n) = b q^n,$$
where $q$ and $b$ are constant natural numbers.

We have to find a closed term representing the sequence defined by the recurrence:
$$a(n+1) = p \min( a(n) ,  bq^n - a(n) ),$$

with $a(0) = a \in \mathbb N$ and 
$ p, q, b \in \mathbb N$ are further constant natural numbers with $a \leq b$ and $p \leq 2q$. 

We observe that $b(n)$ is strictly increasing, and that in general $a(n) \leq b(n)$. This means that for all $k = 0, 1, \dots n$, $a(k) \leq b(n) < b(n) + 1$. We consider the following relations:
\begin{equation}\label{eqbasistent}
    B = b q^n + 1,
\end{equation}
\begin{equation}\label{eqrecurrencetenttent}
    x = x_1 + \alpha B^k + \beta B^{k+1} + x_2 B^{k+2},
\end{equation} 
\begin{equation}\label{eqchoicetent}
    [(2\alpha - bq^k - \gamma)^2 + (\beta - pbq^k + p\alpha)^2]
    [(bq^k - 2\alpha - \gamma)^2 + (\beta - p\alpha)^2] = 0,
\end{equation}
\begin{equation}\label{conditionatent}
    \alpha \leq bq^k,
\end{equation}
\begin{equation}\label{eqremainderconditiontent}
    x_1 < B^k.
\end{equation} 

We observe that if $\alpha \leq pq^k < B$ then $\beta$ will be the numerator of the next sequence term, so it will fulfill the inequality $\beta \leq bq^{k+1} < B$. This inequality has not to be written down explicitly. 

This system can be transformed in an exponential Diophantine equation with parameters $b, p, q, x, n$ and all other variables, including $k$, as unknowns. It is important to understand that $q$, $p$, $b$, $n$, $x$ and, consequently, $B$, are constant. The exponential polynomial has not to be simple in these variables. 

We transform the system in a simple exponential polynomial as follows: 
\begin{equation}\label{eq1tentmap1}
    \left(\alpha + x_3 - qb^k\right)^2 + \left(x_1 + x_4 + 1 - B^k\right)^2 + (x_1 + \alpha B^k + \beta B^{k+1} + x_2 B^{k+2} - x)^2 +
\end{equation}
\begin{equation}\label{eq1tentmap2}
+ [(2\alpha - bq^k - \gamma)^2 + (\beta - pbq^k + p\alpha)^2]
    [(bq^k - 2\alpha - \gamma)^2 + (\beta - p\alpha)^2] = 0
\end{equation}
Call this equation:
$$E(b,p,q,x,n)(k,\alpha, \beta, \gamma, x_1, x_2, x_3) = 0,$$
where $B$ stays for $bq^{n} + 1$. By the hypercube method, there is an arithmetic term 
$$N(b,p,q,x,n)$$
counting the number of solutions $(k, \alpha, \beta, \gamma, x_1, x_2, x_3)$ of this equation in $\mathbb N^7$.

\begin{lemma}\label{lemma:maintent} If $p \leq 2q$ and $a \leq b$ then the system of conditions
$$ y < B^{n+1}\,\, \wedge \,\,y = a + B x_1\,\, \wedge \,\, N(b,p,q,y,n) = n $$
has a unique solution $y$. Moreover, this solution encodes the sequence $a(0),a(1),\dots,a(n)$ as positional code with base $B$. 
\end{lemma} 

{\bf Proof}: By the first condition, $y$ is a code of a sequence of length $n+1$. By the second condition, the first encoded number is $a(0) = a$. Given $y$, in any solution $(k, \alpha, \beta, \gamma, x_1, x_2, x_3)$, $k$ uniquely determines $\alpha$ as the projection of $y$ on place $k$, and uniquely determines the values of the unknowns $x_1, x_2, x_3$. As $k = 0, 1, \dots, n-1$ can take only $n$ values, and there are exactly $n$ solutions, for every value of $k$ there is a solution. But this forces that for $k = 1$ the corresponding $\alpha = a(1)$, for $k = 2$, the corresponding $\alpha = a(2)$, and so on. \qed 

The condition $y < B^{n+1}$ can be written as:
$$\left( 2^{y + x_4 + 1} - 2^{B^{n+1}} \right)^2 = 0,$$
because $B^{n+1}$ consists only in parameters, and purely exponential equations lead to $G_0$-terms, which are much easier than the $G_1$ terms. 
We observe also that $y = a + Bx_1$ is best written as:
$$\left( 2^y - 2^a \cdot {(2^B)}^{x_1} \right)^2 = 0.$$

One transforms the conditions of \cref{lemma:maintent} in an exponential Diophantine equation 
$$F(a, b, p, q, n)(y, \vec x) = 0$$
which has a unique solution $(y, \vec x)$. By replacing $y$ with $\beta + \gamma + 1$, we obtain an exponential Diophantine equation:
$$F(a, b, p, q, n)(\alpha + \beta + 1, \vec x) = 0.$$ 
By the hypercube method, one constructs the corresponding counting term $M(a, b, p, q, n)$ with the property that:
$$M(a, b, p, q, n) = y$$
defined above. It follows that for all $n \in \mathbb N$,
$$a(n) = \left \lfloor \frac{M(a, b, p, q, n)}{\left ( bq^n + 1 \right )^n} \right \rfloor$$ 
and the closed form for the Logistic Map is
$$x(n) = \frac{a(n)}{b(n)},$$
where this time the division is to be understood as rational fraction.  Sometimes, this fraction could be possibly simplified, and the reduced form can be as well expressed by arithmetic terms, because one has a term computing $\gcd(a,b)$. 

\subsection{The Logistic Map}\label{Subsect.Logistic} 

The Logistic Map is defined by the recurrence $x(n+1) = f_r(x(n))$ where

$$f_r(x) = r x(1-x), $$

with given $x(0) \in [0,1]$ and $r \in [0, 4]$, see \cite{May}. We see that all elements of the sequence belong to $[0,1]$. Indeed, by the means' inequality, if $x \in [0,1]$ then 
$$f_r(x) = rx(1-x)\leq 4 \left(\frac{x + (1-x)}{2}\right)^2 = 1.$$

In order to represent the Logistic Map by arithmetic terms, we consider the case with $x(0), r \in \mathbb Q$. 

Let $x(0) = a/b$ and $r = p/q $.

If we write $x(n) = a(n) / b(n)$ in natural numbers $a(n)$ and $b(n)$.

The recurrence rule becomes: 

$$\frac{a(n+1)}{b(n+1)}    = \frac{p}{q} \cdot \frac{a(n)}{b(n)} \cdot \left ( 1 - \frac{a(n)}{b(n)} \right ).$$

We identify numerators and denominators and we obtain the 
crossed recurrence rules:
$$a(n+1) = p a(n) ( b(n) - a(n) ),$$
$$b(n+1) = q b(n)^2.$$
for two sequences of natural numbers, with initial conditions $a(0) = a$ and $b(0) = b$. 

For the sequence $b(n)$ we get the following arithmetic term:
$$b(n) = q^{2^n - 1 } b^{2^n},$$
where $q$ and $b$ are constant natural numbers. Indeed, $b(0) = q^{2^0 - 1 } b^{2^0} = b$ and $b(n+1) = q q^{2^{n+1} - 2 } b^{2^{n+1}}$.

We have to find a closed term representing the sequence defined by the recurrence:
$$a(n+1) = p a(n) (  q^{ 2^n - 1 } b^{2^n} - a(n) )$$

with $a(0) = a \in \mathbb N$ and 
$ p, q, b \in \mathbb N$ are further constant natural numbers with $a \leq b$ and $p \leq 4q$. 

We observe that $b(n)$ is strictly increasing, and that in general $a(n) \leq b(n)$. This means that for all $k = 0, 1, \dots n$, $a(k) \leq b(n) < b(n) + 1$. We consider the following relations:
\begin{equation}\label{eqbasis}
    B = q^{2^n - 1} b^{2^n} + 1,
\end{equation}
\begin{equation}\label{eqrecurrence}
    x = x_1 + \alpha B^k + p \alpha ( q^{2^k - 1} b^{2^k} - \alpha ) B^{k+1} + x_2 B^{k+2},
\end{equation} 
\begin{equation}\label{conditiona}
    \alpha \leq q^{2^k - 1} b^{2^k},
\end{equation}
\begin{equation}\label{eqremaindercondition}
    x_1 < B^k.
\end{equation} 

We observe that if $\alpha \leq q^{2^k - 1} b^{2^k} < B$ then $\beta$ will be the numerator of the next sequence term, so it will fulfill the inequality $ p \alpha ( q^{2^k - 1} b^{2^k} - \alpha ) \leq q^{2^{k+1} - 1} b^{2^{k+1}} < B$. This inequality has not to be written down explicitly. 

This system can be transformed in an exponential Diophantine equation with parameters $b, p, q, x, n$ and all other variables, including $k$, as unknowns. It is important to understand that $q$, $p$, $b$, $n$, $x$ and, consequently, $B$, are constant. The exponential polynomial has not to be simple in these variables. 

We transform the system in a simple exponential polynomial as follows: 
\begin{equation}\label{eq1logisticmap1}
    \left(x_3 + 1 -2^k\right)^2 + \left(\alpha + x_4 - q^{x_3}b^{x_3+1}\right)^2 + \left(x_1 + x_5 + 1 - B^k\right)^2 +
\end{equation}
\begin{equation}\label{eq1logisticmap2}
+ \left(x_1 + \alpha B^k + p \alpha ( q^{x_3} b^{x_3 + 1} - \alpha ) B^{k+1} + x_2 B^{k+2} - x\right)^2 = 0
\end{equation}
Call this equation:
$$E(b,p,q,x,n)(k,\alpha,x_1, x_2, x_3,x_4,x_5) = 0,$$
where $B$ stays for $q^{2^n - 1} b^{2^n} + 1$. By the hypercube method, there is an arithmetic term 
$$N(b,p,q,x,n)$$
counting the number of solutions $(k, \alpha, x_1, \dots, x_5)$ of this equation in $\mathbb N^7$. 

\begin{lemma}\label{lemma:mainlogistic} If $p \leq 4q$ and $a \leq b$ then the system of conditions
$$y < B^{n+1} \,\, \wedge \,\,y = a + B x_1\,\, \wedge \,\, N(b,p,q,y,n) = n $$
has a unique solution $y$. Moreover, this solution encodes the sequence $a(0),a(1),\dots,a(n)$ as positional code with base $B$. 
\end{lemma} 

\begin{proof} By the first condition, $y$ is a code of a sequence of length $n+1$. By the second condition, the first encoded number is $a(0) = a$. Given $y$, in any solution $(k, \alpha, x_1, \dots, x_5)$, $k$ uniquely determines $\alpha$ as the projection of $y$ on place $k$, and uniquely determines the values of the unknowns $x_1, \dots, x_5$. As $k = 0, 1, \dots, n-1$ can take only $n$ values, and there are exactly $n$ solutions, for every value of $k$ there is a solution. But this forces that for $k = 1$ the corresponding $\alpha = a(1)$, for $k = 2$, the corresponding $\alpha = a(2)$, and so on. 
\end{proof}

Again, we better work with: 
$$\left( 2^{y + x_4 + 1} - 2^{B^{n+1}} \right)^2 = 0,$$
$$\left( 2^y - 2^a \cdot {(2^B)}^{x_1} \right)^2 = 0.$$
for the first two relations. 

One transforms the conditions of \cref{lemma:mainlogistic} in an exponential Diophantine equation 
$$F(a, b, p, q, n)(y, \vec x) = 0$$
which has a unique solution $(y, \vec x)$. By replacing $y$ with $\beta + \gamma + 1$, we obtain an exponential Diophantine equation:
$$F(a, b, p, q, n)(\alpha + \beta + 1, \vec x) = 0.$$ 
By the hypercube method, one constructs the corresponding counting term $M(a, b, p, q, n)$ with the property that:
$$M(a, b, p, q, n) = y$$
defined above. It follows that for all $n \in \mathbb N$,
$$a(n) = \left \lfloor \frac{M(a, b, p, q, n)}{\left ( q^{2^n - 1} b^{2^n} + 1 \right )^n} \right \rfloor$$ 
and the closed form for the Logistic Map is
$$x(n) = \frac{a(n)}{b(n)},$$
where this time the division is to be understood as rational fraction.  Sometimes, this fraction could be possibly simplified, and the reduced form can be as well expressed by arithmetic terms, because one has a term computing $\gcd(a,b)$. 

\subsection{Collatz' Sequence}\label{Subsect:Collatz}

Instead of the original Collatz' Sequence, see \cite{Collatz}, given by the recurrence:
$$ a(n+1) = \begin{cases} 3a(n) + 1, & a(n) \textrm{ odd,}\\
a(n)/2, & a(n) \textrm{ even,}\end{cases}$$
it is easier to construct an arithmetic term for its subsequence of odd terms. Its recurrence reads:
$$a(n+1) = \frac{3a(n) + 1}{2^{\nu_2(3a(n)+1)}}.$$ 
We put the condition that $a(0)$ is always odd. We observe that Collatz' Conjecture about the existence of an $n$ such that $a(n) = 1$ for the original sequence is equivalent with the same statement for the modified sequence. We intend to construct an arithmetic term $CS(a,n)$ which values $a(n)$, the $n$-th term of the sequence with $a(0) = a$. 

\begin{lemma}
    For the odd Collatz' Sequence, one has for all $n \in \mathbb N$ that $a(n+1) \leq 2a(n)$. 
\end{lemma} 

\begin{proof}
    As $a(n)$ is odd, one also has $a(n) \geq 1$. Then:
    $$a(n+1) \leq \frac{3a(n) + 1}{2} \leq \frac{4a(n)}{2} = 2a(n).$$
\end{proof}

If we want to encode the finite sequence $a(0), a(1), \dots , a(n)$ positionally, the base $B = 2^{n+1}a(0) = 2^{n+1}a$ is sufficiently large for this encoding. We consider the following equations:

\begin{equation}\label{eqbasiscollatz}
    B = 2^{n+1}a,
\end{equation}
\begin{equation}\label{eqrecurrencecolatz}
    x = x_1 + \alpha B^k +  \beta B^{k+1} + x_2 B^{k+2},
\end{equation} 
\begin{equation}\label{eqrulecollatz1}
    3\alpha + 1 = 2^\gamma \beta,
\end{equation}
\begin{equation}\label{eqrulecollatz2}
    \beta = 2\delta + 1
\end{equation}
\begin{equation}\label{conditioncbase}
    \alpha < B,
\end{equation}
\begin{equation}\label{eqremainderconditioncollatz}
    x_1 < B^k.
\end{equation} 

This system can be transformed in an exponential Diophantine equation with parameters $a, x, n$. All other variables, including $k$, are unknowns. It is important to understand that $a$, $x$, $n$ and $B$, are constant. The exponential polynomial has not to be simple in these variables. 

$$( x -  x_1 - \alpha B^k -  \beta B^{k+1} - x_2 B^{k+2}  )^2 + (3\alpha + 1 - 2^\gamma \beta )^2 + (2^\beta - 2^{2\delta + 1} )^2 + (2^{\alpha + \varepsilon + 1}  - 2^B )^2 + (2^{x_1 + x_3 + 1} - 2^{B^k})^2 = 0. $$ 

Call this equation:
$$E(a, x, n)(\alpha. \beta, \gamma, \delta, \varepsilon, x_1, x_2, x_3) = 0 $$
where $B$ stays for $2^{n+1} a$. By the hypercube method, there is an arithmetic term
$$P(a, x, n)$$
counting the number of solutions of this equation in $\mathbb N^8$.

\begin{lemma}\label{lemma:maincollatz} The system of conditions
$$y < B^{n+1} \,\, \wedge \,\,y = a + B x_1\,\, \wedge \,\, P(a, x, n) = n $$
has a unique solution $y$. Moreover, this solution encodes the sequence $a(0),a(1),\dots,a(n)$ as positional code with base $B$. 
\end{lemma} 

\begin{proof}
    Same proof as for the \cref{lemma:mainlogistic}.
\end{proof} 

Also here, we better work with: 
$$\left( 2^{y + x_4 + 1} - 2^{B^{n+1}} \right)^2 = 0,$$
$$\left( 2^y - 2^a \cdot {(2^B)}^{x_1} \right)^2 = 0.$$
for the first two relations. 

One transforms the conditions of \cref{lemma:maincollatz} in an exponential Diophantine equation 
$$F(a, n)(y, \vec x) = 0$$
which has a unique solution $(y, \vec x)$. By replacing $y$ with $\beta + \gamma + 1$, we obtain an exponential Diophantine equation:
$$F(a, n)(\alpha + \beta + 1, \vec x) = 0.$$ 
By the hypercube method, one constructs the corresponding counting term $Q(a, n)$ with the property that:
$$Q(a, n) = y$$
defined above. It follows that for all $n \in \mathbb N$,
$$a(n) = \left \lfloor \frac{Q(a, n)}{\left ( 2^{n+1} a \right )^n} \right \rfloor$$ 
is the $n$-th term of the odd Collatz Sequence.

\subsection{Visualizing a Julia set}\label{Subsect.Julia}

In this section we develop an arithmetic term for deciding the color of point of the plane for visually approximating a Julia set, see Gaston Julia's original paper \cite{Julia}. We restrict ourselves to the following:

{\bf Algorithm}: {\it Given a circle of radius $R$ centered in $0 \in \mathbb C$ and a constant $c \in \mathbb C$ with $R^2 - R \geq |c|$. Consider the sequence given by the recurrence rule $z(n) = z(n-1)^2 + c$, with $z(0) \in \mathbb C$ such that $|z(0)| < R$. Choose some $N \in \mathbb N$. Compute the elements $z(n)$ for $0 \leq n \leq N$ until $|z(n)| \geq R$. If this condition is fulfilled, output $n$ as color of $z(0)$. If $z(N)$ is still in the open disk, output $N$ as color of $z(0)$.} 

\begin{lemma}
    If $|c| \leq R^2 - R$ and $|z| \geq R$, then $|z^2 + c| \geq R$.
\end{lemma}

\begin{proof}
    $$|z^2 + c | \geq |z^2| - |c| = |z|^2 -|c| \geq R^2 - (R^2 - R) = R.$$
\end{proof} 

This means that once escaped from the open disk, the sequence will never come back. 

As for the Logistic Map and for the Tent Map, it make sense to consider only points with rational coordinates, so to consider only elements of the field $\mathbb Q[i]$, where $i$ is the imaginary unit. Moreover, we need three integers to define such a complex point, which we write as:
$$z = \frac{x + iy}{u}$$
with $x, y, u \in \mathbb Z$, $u \neq 0$. Also, we can represent arbitrary integers only as differences of natural numbers:
$$z = \frac{(x_+ - x_-) + i(y_+ - y_-)}{u}.$$
Observe that the denominator can be always chosen to be positive. Such representation is far from being unique, but this should not represent a problem. Instead of computing one sequence $z(n)$ of elements of $\mathbb Q[i]$, we compute five sequences of natural numbers $x_+(n)$, $x_(n)$, $y_+(n)$, $y_-(n)$, $u(n)$ such that for all $n \in \mathbb N$,
$$z(n) = \frac{(x_+(n) - x_-(n)) + i(y_+(n) - y_-(n))}{u(n)},$$
with $u(0) = u$, $x_\pm(0) = x_\pm$, $y_\pm(0) = y_\pm$. Also, let $c$ be
$$c = \frac{(w_+ - w_-) + i(v_+ - v_-)}{t}.$$
We observe that 
$$z^2 = \frac{x_+^2 + x_-^2-2 x_+x_- -y_+^2 - y_-^2 + 2y_+y_- + 2i ( x_+ y_+ + x_- y_- - (x_+ y_- + x_- y_+) ) }{u^2} = $$
$$ = \frac{[(x_+^2 + x_-^2 + 2 y_+ y_-) - ( y_+^2 + y_-^2 + 2 x_+x_- ) ] + 2i [(x_+y_+ + x_- y_-) - (x_+ y_- + x_- y_+)]}{u^2}.$$
This leads to the following cross-recurrence rules:
\begin{eqnarray*}
 x_+(n+1) &=& t(x_+(n)^2 + x_-(n)^2 + 2y_+(n)y_-(n)) + u(n)^2w_+, \\
 x_-(n+1) &=& t( y_+(n)^2 + y_-(n)^2 + 2 x_+(n)x_-(n) ) + u(n)^2w_-,\\
 y_+(n+1) &=& 2t(x_+(n)y_+(n) + x_-(n) y_-(n)) + u(n)^2 v_+,\\
 y_-(n+1) &=& 2t (x_+(n) y_-(n) + x_-(n) y_+(n)) + u(n)^2 v_-, \\
 u(n+1) &=& t u(n)^2 
\end{eqnarray*}

Following the computation in \cref{Subsect.Logistic}, $u(n) = t^{2^n-1}u^{2^n}$ for all $n \in $

If $$A = 5(x_+ + x_- + y_+ + y_- + u + w_+ + w_- + v_+ + v_- + t + 1)$$
we see that for all $n \in \mathbb N$,
$$x_+(n), x_-(n), y_+(n), y_-(n), u(n) < A^{3^n}.$$
So we can take $B = A^{3^n}$ as a base to positionally represent the sequence:
$$x_+(0), x_-(0), y_+(0), y_-(0), u(0), \dots, x_+(n), x_-(n), y_+(n), y_-(n), u(n). $$

For some number $x$, we write down an equation telling that the digits in base $B$ corresponding to $x_+(k+1)$, $x_-(k+1)$, $y_+(k+1)$, $y_-(k+1)$, $u(k+1)$ depend on the digits corresponding to $x_+(k)$, $x_-(k)$, $y_+(k)$, $y_-(k)$, $u(k)$ as the recurrence rules say. An arithmetic term:
$$T(x, w_+, w_-, v_+, v_-, t)$$
counts the number of values of $k$ for which $x$ satisfies this conditions. The condition:
$$(x_+,x_-,y_+y_-,u)(0) = (x_+,x_-,y_+y_-,u) \wedge T(x,w_+, w_-, v_+, v_-, t) = n \wedge x < B^{5(n+1)} $$
is satisfied by only one natural number $x$, namely the natural number encoding the aforementioned sequence. By replacing $x$ with $a + b + 1$, we construct the counting solutions term:
$$W(n, x_+,x_-,y_+,y_-,u, w_+, w_-, v_+, v_-, t)  $$
which computes $x$. Now suppose that the escape radius
$$R = \frac{a}{b}$$
is a positive rational number. One can construct a counting term
$$V(a,b,x, B)$$
which says for how many segments of length $5$ in $x$, the complex number $z_k$ represented by the tuple of natural numbers $x_+(k), x_(k), y_+(k), y_-(k), u(k)$ satisfies $|z(k)| < R$. It follows that the term:
$$V(a,b,W(N, x_+,x_-,y_+,y_-,u, w_+, w_-, v_+, v_-, t),B)  $$
computes the color of the point $z(0)$ in the corresponding Julia set. If one puts the initial term $z(0) = 0$ and considers the point $c$ as a variable, the same arithmetic term produces the coloring of the Mandelbrot set:
$$V(2,1,W(N, 0,0,0,0,1, w_+, w_-, v_+, v_-, t),B). $$ 

\section{A Turing complete term} \label{Section.TCT} 

A further application of \cref{metatheorem} will be a Turing complete arithmetic term. Turing machines are ideal computation devices defined the first time by Alan Turing in his seminal paper \cite{Turing}. Church Thesis states that the correct notion of a computable function means computable by a Turing machine. We explain now what does it mean to apply \cref{metatheorem} for Turing machines. Suppose one has a deterministic Turing machine $T$. Let us encode some Turing machine configuration $c$ in a natural number $n(c)$ in some canonic computable way. Then there is a computable function $f: \mathbb N \rightarrow \mathbb N$ such that $f(n(c)) = n(c')$ if and only if the configuration $c'$ follows from $c$ in one step. Moreover, the function $f$ can be expressed by an arithmetic term, being a finite case-disjunction. Let $n_0$ be the natural number expressing the start configuration of $T$. The sequence with initial term $n_0$ and recurrence rule:
$$n_{i+1} = f(n_i)$$
encodes the sequence of successive configurations of the Turing machine. By our \cref{metatheorem}, there is an arithmetic term $F(n)$ such that for all $i \in \mathbb N$,
$$n_i = F(i).$$
We will see in this section that one can construct a term $F(\tau, c_0, i)$ such that for any deterministic Turing machine encoded by $\tau$, for every start configuration $c_0$ and for every $i \in \mathbb N$, $F(\tau, c_0, i)$ is the configuration of this Turing machine after step number $i$, with $F(\tau, c_0, 0) = c_0$. 

Suppose now, that the machine $T$ will stop after a number of steps which is bounded by a number $t(n_0)$, and the function $t(n)$ is computable. The configurations and the transition function can be normed such that if $c$ is a final configuration, then the next configuration is $c' = c$. If the Turing machine computes some function $\varphi(n)$, then we can extract the value $\varphi(n)$ from the final configuration $F(\tau, c_0(n), t(c_0(n))$, and this extraction can be done composing $F$ with an arithmetic term. Finally, we will get a uniform way to express all Kalmar elementary functions by a universal arithmetic term. 

We state the following specifications:
 
\begin{enumerate} 
    \item For simplicity we are modeling deterministic Turing machines with one tape  which {\it never step into the negative side of the tape}. This can be done by adding some new letters and some new states, see Arora and Barak \cite{AroraBarak}, Chapter 1. 
    \item The machine is $(Z, \Gamma, z_0, \Box, E)$. We will always encode $z_0$ and the blank symbol $\Box$ with $0$.
    \item The machine always starts in state $z_0$. The head sees the first symbol of the input in cell nr. $0$. 
    \item We build a term $T(c, d, P, x, l, n)$ which computes the code of the final configuration of the Turing machine with program encoded in $P$ for the input encoded in $x$ after running $n$ steps. This will be the Turing Universal Term. The integers $c$ and $d$ are only  auxiliary variables to read the program $P$. Namely, the number of states $|Z|$ and the number of tape characters $|\Gamma|$ are both strictly smaller than $2^c$.  The number $l$ is the length of the input. This can be found by various exponential Diophantine equations, or by directly using the term $\lfloor \log_{2^c}  x \rfloor$, but the number $l$ will be a constant when making a computation. So, as long as we look for an arithmetic term, which is not so complicated, we may take it as an input variable. Also, $\delta$ consists of $d$ many lines. Possibly $d = |Z| \cdot |\Gamma|$. 
    \item For getting $T(c, d,  P, x, l, n)$ we construct another term $R(c, d, P, x, l,  n)$ which produces a number $y$ encoding positionally the sequence of all $n+1$ configurations of the machine, from the start configuration to the configuration after the step number $n$. We call $R(c, d, P, x, l,  n)$ the {\bf run} of the machine. 
    \item The program $P$ is a number that positionally encodes the function $\delta$ using the numeration basis $2^c$. A $\delta$-instruction is a $5$-tuple
    $$(z, a, z', a', m)$$
    where $\delta(z,a) = (z', a', m)$, $z, z' \in Z$, $a,a' \in \Gamma$, $0 \leq m \leq 2$. We always take $c \geq \max(2, |Z|, |\Gamma|)$. In the numeration base $2^c$, $P$ has length $5|\delta| = 5d$. So $P < 2^{5cd}$. We interpret $m=0$ as no move, $m = 1$ as move one cell to the right and $m = 2$ as move one cell to the left. For $0 \leq j \leq d$ one has $\pi(P, 2^c, 5j) = 2^z$, $\pi(P, 2^c, 5j+1) = 2^a$, etc.
    \item For $0 \leq k \leq n$, the machine configuration number $k$ is the triple:
    $$\iota_k = (z_k, p_k, x_k).$$
    Here $z_k$ is the state of the Turing machine after $k$ steps, $p_k$ is the position of the head after $k$ steps, and $x_k$ is the code of the tape content after $k$ steps, written as number in basis $2^c$. As the blank symbol is encoded by $0$, this content is always a natural number. We observe that $\iota_0 =(0,0,x)$, where $x$ is the input of the machine. We observe that the length of $x_k$ is $\leq l + n$ so the number $x_k < 2^{c(n+l)}$. We will use the base $2^{c(n+l)}$ to positionally encode the machine configurations. Observe that the number $c(n+l)$ will be composed by parameters in the first equation, so is constant. 
    \item The word $w = a_0 \dots a_t$ is expressed by the number:
    $$ \sum _{k=0}^t 2^{ck + a_k}.$$
    \item The positional code of the sequence of machine configurations $\iota_0, \dots, \iota_n$ is a number $y < 2^{3cn(n+l)}$. 
\end{enumerate} 

The following conditions say that the numbers $v$ and $v'$ encode possible tape-contents, the numbers $z$ and $z'$ encode possible states and the numbers $a$ and $a'$ encode possible elements of the alphabet $\Gamma$:
\begin{eqnarray}\label{conditions}
    v &<& 2^{c(n+l)} ,\\
    v'&<& 2^{c(n+l)} ,\\
    2^z &<& 2^c ,\\
    2^{z'} &<& 2^c,\\
    2^a &<& 2^c,\\
    2^{a'} &<& 2^c.
\end{eqnarray}
The following conditions say that the $y$ contains two possible machine-configurations $\iota_k = (z, p, v)$ and $\iota_{k+1} = (z', p', v')$, and that they occur consecutively. The states $z, z'$ are written as $2^z, 2^{z'}$, and the same is done with the letters $a, a'$. The positions $p, p' \leq n < c(n +l)$ so they can be also represented by the numbers $2^p, 2^{p'}$. This is expressed by the equation:
\begin{equation}\label{eqtwoconfigurations1}
    y = x_2 + 2^{3kc(n+l) + z} + 2^{(3k+1)c(n+l) + p} + 2^{(3k+2)c(n+l)}\cdot v + 
\end{equation} 
\begin{equation}\label{eqtwoconfigurations2}
    + 2^{(3k+3)c(n+l) + z'} + 2^{(3k+4)c(n+l) + p'} + 2^{(3k+5)c(n+l)}\cdot v'+
    2^{(3k+6)c(n+l)}\cdot x_3 
\end{equation} 
together with the condition:
\begin{equation}\label{eqconditionx2}
    x_2 < 2^{3kc(n+l)}
\end{equation}
The following equation, together with the next condition,  says that the head sees the letter $a \in \Gamma$:
\begin{equation}\label{eqreadtheletter}
    v = x_4 + 2^{pc+a} + 2^{(p+1)c} \cdot x_5,
\end{equation} 
\begin{equation}\label{eqconditionx4}
    x_4 < 2^{pc}.
\end{equation}
The following equation, together with the following condition, expresses the fact that the program-line $\delta(z,a) = (z', a', m)$ is in the program at line $j$:
\begin{equation}\label{eqprogramline}
    P = x_6 + 2^{5jc+z} + 2^{(5j+1)c+a} + 2^{(5j+2)c+z'} + 2^{(5j+3)c+a'} + 2^{(5j+4)c+m}+ 2^{5(j+1)c} \cdot x_7,
\end{equation} 
\begin{equation}\label{eqconditionx6}
    x_6 < 2^{5jc},
\end{equation} 
\begin{equation}\label{eqconditionj}
    j < d.
\end{equation}
But $p$ and $p'$ are related by the movement $m$ as follows:
\begin{eqnarray}\label{nextpositionexplicitly}
[m^2 + (p'-p)^2][(m-1)^2 + (p'- p - 1)^2][(m-2)^2 + (p' - p + 1)^2] = 0.
\end{eqnarray} 
The new tape-content should be:
\begin{eqnarray}\label{nextcontentexplicitly}
    v' &=& v - a \cdot 2^{pc} + a' \cdot 2^{pc}.
\end{eqnarray} 

This system of equations will now be rewritten such that: 

- Every equation becomes an equality where the left-hand side is $\geq 0$.

- As much as possible is expressed by exponential monomial instead of polynomial monomials. 

- We do not introduce more new unknowns as necessary. 

\begin{eqnarray}\label{conditionssquared}
    (v + x_8 + 1 -  2^{c(n+l)})^2 &=& 0 ,\\
    (v' + x_9 + 1 -  2^{c(n+l)})^2 &=& 0 ,\\
    (2^z + x_{10} + 1 - 2^c)^2 &=& 0 ,\\
    (2^{z'} + x_{11} + 1 - 2^c)^2 &=& 0 ,\\
    (2^{a} + x_{12} + 1 - 2^c)^2 &=& 0 ,\\
    (2^{a'} + x_{13} + 1 - 2^c)^2 &=& 0 ,\\
    (x_2 + x_{14} + 1 - 2^{3kc(n+l)})^2 &=& 0 ,\\
     (x_4 + x_{15} + 1 - 2^{pc})^2 &=& 0 ,\\ 
     (x_6 + x_{16} + 1 - 2^{5jc})^2 &=& 0 ,\\
     (2^j + x_{17} + 1 - 2^d)^2 &=& 0,
\end{eqnarray}
\begin{equation}\label{eqtwoconfigurations3}
    [x_2 + 2^{3kc(n+l) + z} + 2^{(3k+1)c(n+l) + p} + 2^{(3k+2)c(n+l)}\cdot v + 
\end{equation} 
\begin{equation}\label{eqtwoconfigurations4}
    + 2^{(3k+3)c(n+l) + z'} + 2^{(3k+4)c(n+l) + p'} + 2^{(3k+5)c(n+l)}\cdot v'+
    2^{(3k+6)c(n+l)}\cdot x_3 - y ]^2 = 0
\end{equation} 
\begin{equation}\label{eqreadtheletter1}
    ( x_4 + 2^{pc+a} + 2^{(p+1)c} \cdot x_5 - v)=0,
\end{equation} 
\begin{equation}\label{part3}
[(2^m - 1)^2 + (2^{p'}-2^{p})^2][(2^m-2)^2 + (2^{p'}- 2^{p + 1})^2][(2^m-4)^2 + (2^{p'+1} - 2^p)^2] = 0,
\end{equation} 
\begin{equation}\label{part4}
    (x_{1} - 2^{a + pc})^2 + (x_{18} - 2^{a'+pc})^2 + (2^{v'+x_1} - 2^{v + x_{18}})^2 = 0.
\end{equation} 

Parameters: $c, d, P, l, n, y$.

Unknowns: $k, z, p, v, a, j,  z', a', m, p', v', x_1, \dots, x_{18} $. 

We denote this exponential Diophantine equation in $6$ parameters and $29$ unknowns
$$E(c, d, P, l, n, y)( k, \vec \lambda) = 0.$$

By Mazzanti, there is an arithmetic term
$N(c, d, P, l, n, y)$ computing the number of tuples $(k, \vec \lambda)$ which are solutions of the equation.  

\begin{lemma} If $P < 2^{5cd}$ is a correct determinist program of length $d$ of a Turing machine, and if $x$ of length $l$ is the input of this machine, then the system of conditions
$$y < 2^{3c(n+1)(n+l)} \,\, \wedge \,\,y = 2^{2c(n+l)} x + 2^{3c(n+l)} x_1\,\, \wedge \,\, N(c, d, P, l, n, y) = n $$
has a unique solution $y$. Moreover, this solution encodes the sequence of all machine configurations from $\iota_0 = (0,0,x)$ to $\iota_n = (z_n, p_n, x_n)$. 
\end{lemma}

\begin{proof} By the length condition, $y$ can encode at most $n+1$ configurations. By the initial condition given in the statement, the stack of configurations starts with $\iota_0$, which is the start configuration of the Turing machine. One cannot have two different solutions with the same $k$, because $k$ determines uniquely $p, z, v$, $p$ and $v$ determine uniquely $a$, $z$ and $a$ determine uniquely $z'$, $a'$ and $m$, and these values determine uniquely $p'$ and $v'$. Also, the values of the auxiliary unknowns $x_{1}, \dots, x_{18}$ are uniquely determined by these. As there are exactly $n$ solutions and as $0 \leq k \leq n-1$ by the length condition, for any of these values of $k$ there is a unique solution representing the transition of the deterministic Turing machine from the configuration $\iota_k$ to the configuration $\iota_{k+1}$. By induction, all these configurations are uniquely determined and are exactly the Turing machine configurations at step $k$ when starting with input $x$.
\end{proof}

So we transform these conditions in an exponential Diophantine equation where we also replace $y$ with $\alpha + \beta + 1$. For this equation:
$$F(c, d, P, x, l, n, \alpha + \beta + 1, \vec \epsilon) = 0$$
we compute by Mazzanti the term $M(c, d, P, x, l, n)$. This term computes the value of $y$. By projecting $y$ on its last segment, we find the configuration after $n$ steps. 

Farther one writes down the term $N(c, d, P, l, n, y)$ corresponding to the equation constructed above, then one builds the equation $F(c, d, P, x, l, n, \alpha + \beta + 1, \vec \epsilon) = 0$ equivalent to:
$$\alpha + \beta + 1 < 2^{3cn(n+l)} \,\, \wedge \,\, \alpha + \beta + 1 = 2^{2c(n+l)} x + 2^{3c(n+l)} x_1\,\,\wedge \,\, F(c, d, P, l, n, a + b + 1) = n, $$
and one constructs the term:
$$M(c, d, P, x, l, n) $$
that computes the number $y$, which is the code of the succession of the configurations $\iota_0, \iota_1, \dots, \iota_n$.

 Now we speak about applications of this term. We consider normed deterministic Turing machines with the following conventions:
 \begin{enumerate}
     \item The machine starts on the first digit of the input, in cell number $0$, being in the state $z_0$.
     \item The machine ends on the first digit of the output, in cell number $0$. No other symbols are written on the tape, excepting the output. 
     \item All final states are states $z_e$ such that for all letters $\gamma \in \Gamma$,
     $$\delta(z_e, \gamma) = (z_e, \gamma, 0).$$
     So, if $n$ is bigger or equal the computation time for the output, than the tape content of the machine configuration after $n$ steps will be the output. 
 \end{enumerate}

 We observe that the term:
 $$O(c, d, P, x, l, n) := \left \lfloor \frac{M(c, d, P, x, l, n)}{2^{(3n+2)c(n+l)}} \right \rfloor $$
outputs the tape content after $n$ steps. Suppose that we consider a computable total function 
$$f : \mathbb N^s \rightarrow \mathbb N$$
and that the computation time for input $a \in \mathbb N^s$ is bounded by a function $T(a)$. Let $P$ be the code of the corresponding Turing machine. The parameter $d$ can be computed from $P$ by an arithmetic term $d(P,c)$, differently from $c$, which still can be given. This can be also repaired by encoding the pair $(c, P)$ in a number $C$ using Cantor's pairing. Also, the input of the machine can be computed by an arithmetic term $x(a, c)$ and its length can be computed by a term $l(x, c)$. The output of the machine can be translated into numbers by another arithmetic term $b(o)$. Altogether our function has the arithmetic expression:
$$f(a) = b(O(c, d(P), P, x(a,c), l(x(a,c),c), T(a)). $$
If we also encode the pair in $C = [c\mid  P]$ by Cantor's, we find the expression:
$$f(a) = b(O(L(C), d(R(C)), R(C), x(a,L(C)), l(x(a,L(C)),L(C)), T(a)) = \Upsilon(C, a, T(a)),$$
where $\Upsilon(C, a, t)$ is the Turing complete term. So we proved:

\begin{theorem}
    There is a Turing complete term $\Upsilon(C, a, T(a))$ which gets a program $C$ to compute a function $f$, an input $a$ and an upper bound $T(a)$ for the computation time of $C$ for input $a$ and computes in a finite and fixed number of arithmetic operations the value $f(a)$. 
    In particular, a function is Kalmar elementary if and only if its computation time is Kalmar elementary. 
\end{theorem}

A very convenient case arises for Kalmar functions $f : \mathbb N \rightarrow \mathbb N$ if we apply the fact that there is always a Turing machine able to input the argument $a \in \mathbb N$ written unary and to output the result $f(a)$ also written unary. Recall that the unary writing means only to repeat one symbol $a$ times, for example $6_{(1)} = 111111$. The symbol $"1"$ will be encoded as $10\dots 0$, which is a word of length $c$. So the encoding term is:
$$x(a, c) = \sum_{k=0}^{a-1}2^{ck} = \left \lfloor \frac{2^{ca} - 1}{2^c - 1} \right \rfloor,$$
$$l(x(a,c), c) = ac,$$
and the translation of the output is 
$$b(o) = \HW(o) = \left \lfloor \frac{\lfloor \log_2 (o + 1) \rfloor}{c} \right \rfloor.$$
With these details, the universal Turing function of one argument is completely described.

\begin{cor}\label{cor.representability}
Suppose that a function $f : \mathbb N^k \rightarrow \mathbb N$ can be computed by a deterministic Turing machine in time bounded by a function $T: \mathbb N^k \rightarrow \mathbb N $. Then $f$ is expressible by an arithmetic term built up using the functions $x+y$, $x \bmod y$, $2^x$ and $T(\vec x)$. If $T(\vec x)$ is expressible by a classic arithmetic term, say an iterated exponential 
$$T(x_1, \dots, x_k) = 2^{\reflectbox{$\ddots$}^{2^{x_1 + \dots + x_k}}},$$
then the function $f$ can be expressed by a term using only $x+y$, $x \bmod y$ and $2^x$.  
\end{cor}

\section{A simpler universal model: goto programs} \label{Section.Gototerm}

Although the one-tape Turing machine construction of \cref{Section.TCT} proves full universality, actually writing programs in that model is excruciatingly painful. We now lift the simulation to a much more readable and familiar computational model: goto programs operating on registers via increment, decrement, and zero-testing.

The possibility to apply \cref{metatheorem} to register machines can be justified in a similar way as already done for Turing machines. We will define below the notion of configuration of a deterministic register machine. The configuration consists of the content of the $n$ registers and the label of the program-line to be applied next.  Every configuration $c$ of the register machine can be encoded in a natural number $n(c)$ and there is an arithmetic term $f$ such that the configuration $c'$ follows from the configuration $c$ in one step if and only if $f(n(c)) = n(c')$. Like before, it follows that there exist an arithmetic term which computes directly the configuration after $i$ computation steps. In this section, we will show that we can construct a uniform term, which simulates every possible register machine and outputs the configuration after an arbitrary number of steps. 

It is well known that goto programs are Turing complete \cite{Schoning2008}. Each line is labeled, and the instruction set consists of:

$(L)\,\, x_i := x_j + 1;$

$(L)\,\, x_i := x_j \dot{-} 1;$

$(L) \textrm{ if } x_i = 0 \textrm{ goto } L_1;$

$(L) \textrm{ goto } L_2;$

$(L) \textrm{ stop};$ 

Let $m$ be the number of lines of a goto program. All labels vary from $1$ to $m$. Also, there are $n-1$ registers $x_1$, $\dots$, $x_{n-1}$. Any register may contain only a natural number. The configuration of the goto program at a given moment is the vector:
$$\iota = (L, x_1, \dots, x_{n-1})$$
representing the actual label $L$ and the content of the different registers, before the line $L$ is performed. 

Copy commands will be represented as:
$$(L)\,\,(i, j, s, 0, 0, 0).$$
Here $s \in \{0, 1, 2\}$ and the meaning is:
$$(L)\,\,x_i := (x_j + s) \dot{-} 1;$$

Goto commands are represented as:
$$(L)\,\,(0, 0, 0, u, L_1, L_2),$$
and the meaning is 
$$(L)\,\,\textrm{ if } x_u = 0 \textrm{ then goto } L_1 \textrm{ else goto } L_2;$$
If one has $L_1=L_2$, this represents the unconditional goto command. 

The stop command is represented as:
$$(L)\,\,(0, 0, 0, 0, 0, 0).$$

In order to encode a program in a number $P$ we use the base $b = n + m + 1$. The command line:
$$(i, j, s, u, L_1, L_2)$$
is encoded by the number:
$$c = i + j b + s b^2 + u b^3 + L_1 b^4 + L_2 b^5.$$
The program is encoded by the number:
$$P = \sum _{L = 0}^{m-1} c_L b^{6L}.$$
A number $x$ will encode the sequence of all configurations from $\iota_0 = (1, x_1^{(0)}, x_2^{(0)}, \dots, x_n^{(0)})$ to $\iota_t = (L^{(t)}, x_1^{(t)}, x_2^{(t)}, \dots, x_{n-1}^{(t)})$. In order to encode configurations, we need a basis $B$ larger than all labels and all register contents. Because performing one command, the content of some register can increase with at most $1$, $B = m + \max(\iota_0) + t $ is large enough. A configuration 
$$\iota = (L, x_1, \dots, x_{n-1})$$
is well encoded by:
$$c_\iota = L + x_1 B + \dots + x_{n-1} B^{n-1}$$
and the sequence of configurations $\iota_0, \iota_1, \dots, \iota_t$ is encoded by:
$$x = c_{\iota_0} + c_{\iota_1} B^n + \dots + c_{\iota_t} B^{nt}.$$
We will construct an arithmetic term
$T(m,n,P,\max,t)$
that computes $x$, where $\max = \max(\iota_0)$. To this end, we write down exponential Diophantine equations expressing the fact that for some $k$, the segment of $x$ corresponding to the $k$-th configuration and the segment of $x$ corresponding to the $k+1$-th  configuration are related according to the corresponding line of the program $P$.  

The configuration encoded by $c$ verifies the equations:
$$x = y_1 + c B^{nk} + y_2 B^{n(k+1)},$$
$$c < B^n,$$
$$y_1 < B^{nk}.$$

The label $L$ of the program-line of the configuration is given by the equations:
$$c = L + By_3,$$
$$L < b.$$

This program-line is extracted from $P$:
$$ P = y_4 + i b^{6L} + j b^{6L+1} + s b^{6L+2} + u b^{6L+3} + L_1 b^{6L+4} + L_2 b^{6L+5} + y_5 b^{6(L+1)},$$
$$i < b,\,\,j < b,\,\,s < 3,\,\,u < b,\,\,L_1<b,\,\,L_2 < b,$$
$$y_4 < b^{6L}.$$

Let $c'$ be the configuration resulting from $c$ by applying the program-line. This configuration verifies the equations:
$$x = y_6 + c' B^{n(k+1)} + y_7 B^{n(k+2)},$$
$$c' < B^n,$$
$$y_6 < B^{n(k+1)}.$$

The connection between $c$ and $c'$ is a disjunction of the following situations. 
\begin{enumerate}
  \item If $u=0$ and $i \neq 0$, then we are dealing with a copy-command, $x_i := (x_j+s)\dot{-}1$. 
  \begin{enumerate}
    \item If $|i - j| > 0$ then:
    $$ c = y_8 + e B^i + y_9 B^{\min(i,j)+1} + f B^j + y_{10} B^{\max(i,j)+1},$$ 
    $$y_8 < B^{\min(i,j)},\,\,e < B,\,\, y_9 < B^{|i-j|-1},\,\,f < B.$$
    If $f + s > 0$ then:
    $$c' = c - eB^i + (f + s - 1)B^i + 1,$$
    where the last $+1$ increases the label of the current program-line.
    
    If $f+s=0$ then:
    $$c' = c - eB^i + 1.$$

    \item If $i = j$ then:
    $$ c = y_8 + e B^i + y_9 B^{i+1},$$ 
    $$y_8 < B^i,\,\,e < B.$$
    If $e + s > 0$ then:
    $$c' = c + (s-1) B^i + 1.$$
    If $e+s = 0$ then:
    $$ c' = c+1.$$
  \end{enumerate}
  \item If $i = 0$ and $u \neq 0$, then we are dealing with a goto-command. In this case, 
  $$ c = y_8 + e B^u + y_9 B^{u+1},$$ 
  $$y_8 < B^u,\,\,e < B.$$
  If $e = 0$ then:
  $$c' = c - L + L_1.$$
  If $e > 0$ then
  $$c' = c - L + L_2.$$
  \item If $i = 0$ and $u=0$, then we are dealing with a stop command. In this case,
  $$c' = c.$$
\end{enumerate}
We observe that if we denote $D = |i-j|$, $\mu = \min(i,j)$ and $M = \max(i,j)$ then they are solutions of the following Diophantine equations:
$$(i-j-D)^2(j-i-D)^2 = 0,$$
$$(2M - i - j - D)^2 = 0,$$
$$(2\mu +D - i - j)^2 = 0.$$
This equations are better written as:
$$U_1 = \left(2^i - 2^{j+D}\right)^2\left (2^j - 2^{i+D}\right)^2 + \left (2^{2M} - 2^{i+j+D}\right )^2 +
\left (2^{2\mu+D} - 2^{i+j}\right )^2 = 0.$$

Now we are ready to express the first exponential Diophantine equation as a sum of squares. The expressions:
$$b = m + n + 1,$$
$$B = m + \max + t,$$
are only notations introducing the short-cuts $b$ and $B$. They are define only by parameters, so should not be considered unknowns. 

The configuration $c$ satisfies the following equation:
$$U_2 = \left (x - y_1 - c B^{nk} - y_2 B^{n(k+1)}\right )^2 +\left (2 \cdot 2^{c + z_1} - 2^{B^n} \right)^2 + \left (y_1 + z_2 + 1 - B^{nk}\right )^2 = 0.$$
The label $L$ of the program-line of the configuration is the solution $L$ of the equation:
$$U_3 = \left(2^c - 2^L (2^B)^{y_3}\right)^2 + \left(2^{L+z_3} -2^b\right)^2 = 0.$$
This program-line is extracted from $P$:
$$ U_4 = \left( P - y_4 - i b^{6L} - j b^{6L+1} - s b^{6L+2} - k b^{6L+3} - L_1 b^{6L+4} - L_2 b^{6L+5} - y_5 b^{6(L+1)}\right)^2+$$
$$+\left(2^{i+z_4} -2^ b\right)^2 + \left(2^{j+z_5} -2^ b\right)^2 + \left(2^{s+z_5} - 4\right)^2 + \left(2^{k+z_6} - 2^ b\right)^2 + \left(2^{L_1 + z_7} -2^b\right)^2 +\left(2^{L_2 + z_8} -2^b\right)^2 +$$
$$+\left( 2^{y_4+z_9}- 2^{ b^{6L}}\right)^2 =0.$$
Let $c'$ be the configuration resulting from $c$ by applying the program-line. This configuration verifies the equations:
$$U_5 = \left(x - y_6 - c' B^{n(k+1)} - y_7 B^{n(k+2)}\right)^2+ \left(2^{c'+z_{10}} -2^{ B^n} \right)^2 + \left(y_6 + z_{11} +1 -B^{n(k+1)}\right)^2 = 0.$$

Now we write down the disjunctive conditions. 

If we are dealing with a copy-command, $x_i := (x_j+s)\dot{-}1$, and $i \neq j$, then we construct the expression:

  $$I = \left(2^u - 1\right)^2 + \left(2^i - 2^{z_{12} +1}\right)^2 + \left(2^D - 2^{z_{13} +1}\right)^2 + \left(c - y_8 - e B^i - y_9 B^{\mu+1} - f B^j - y_{10} B^{M+1}\right)^2 + $$ $$+\left(y_8 + z_{14}+ 1 - B^\mu\right)^2 + \left(e + z_{15} + 1 - B\right)^2 + (y_9 + z_{16} + 1 - B^{z_{13}})^2 + \left(f + z_{17} + 1 - B\right)^2 + $$ $$ + \left[\left(2^{f+s} - 2^{z_{18}+1}\right)^2  + \left(c' - c + eB^i - (f + s - 1)B^i - 1\right)^2\right] \cdot \left[\left(2^{f+s} - 1\right)^2 + \left(c' - c + eB^i - 1\right)^2\right].$$

If we are dealing with a copy-command, $x_i := (x_j+s)\dot{-}1$, and $i = j$, then:

$$II = \left(2^u - 1 \right)^2 + \left(2^i - 2^{z_{12} +1}\right)^2 + \left(2^D - 1 \right)^2 + \left(c - y_8 - e B^i - y_9 B^{i+1} \right)^2 + $$ $$+\left(y_8 + z_{14}+ 1 - B^i\right)^2 + \left(2^{e + z_{15} + 1} - 2^B\right)^2 + $$ $$ + \left[\left(2^{e+s} - 2^{z_{18}+1}\right)^2  + \left(c' - c - (s - 1)B^i - 1\right)^2\right] \cdot \left[\left(2^{e+s} - 1\right)^2 + \left(2^{c'} - 2^{c + 1}\right)^2\right].$$

If we are dealing with a goto command, we construct:

  $$III = \left(2^i - 1 \right)^2 + \left(2^u - 2^{z_{12}+1}\right)^2 + \left(c - y_8 - e B^u - y_9 B^{u+1}\right)^2 + \left(y_8 + z_{14}+ 1 - B^u\right)^2 + $$ $$ +\left(2^{e + z_{15} + 1} - 2^B\right)^2 +
  \left[\left(2^e - 1\right)^2 + \left ( 2^{c'+L} - 2^{c+L_1}\right )^2 \right] \cdot \left[\left (2^e - 2^{z_{16} +1} \right)^2 + \left ( 2^{c'+L} - 2^{c+L_2}\right )^2 \right] .$$ 

For stop commands, we construct:

$$IV = \left(2^i - 1 \right)^2 + \left(2^u-1\right)^2 + \left (2^c - 2^{c'}\right )^2 $$ 

Consider the equation:
$$U_6 = I \cdot II \cdot III \cdot IV = 0.$$ 

We consider also the equation:
$$U_7 = \left( 2^{B^{n(t+1)}} - 2^{x + z_{19} + 1}\right)^2 = 0.$$

This equation expresses the fact that $x < B^{n(t+1)}$.
To sum up, consider the equation:

$$E = U_1 + U_2 + U_3 + U_4 + U_5 + U_6 + U_7= 0.$$

This expression depends on the following quantities:
$$E = E(m,n,P,x, \max,t)(k,c,c',L,i,j,s,u,L_1,L_2,y_1,\dots, y_{10},z_1,\dots, z_{17}).$$ 

The corresponding solution counting term:
$$N(m,n,P,x,\max,t)$$
essentially counts the number of values of $k$ with $0 \leq k \leq t$ such that the $k$-th configuration encoded in $x$ and the $(k+1)$-th configuration encoded in $x$ correspond to the program-line encoded in the $k$-th configuration. 

Consider a non-negative exponential Diophantine equation equivalent with:
$$x \bmod (B^{m+1}) = c_0\,\,\wedge \,\,  N(m,n,P,x,\max,t) = t.$$ 

Again, the first condition should be written as:
$$\left( 2^x - 2^{c_0 + B^{m+1}y} \right)^2 =0,$$
taking advantage of the fact that $B$ and $B^{m+1}$ consist of parameters only. 

Here $x$ is the unique unknown, while $(m,n,P, c_0, \max,t)$ are parameters. 

Any $x$ satisfying this equation can be seen as a sequence of $t+1$ configurations of a register machine. The first configuration is the starting configuration $c_0$ of our goto-program $P$, while any two successive configurations are so that one always obtains $c_{k+1}$ from $c_k$ by applying the determinist program $P$. It follows that this equation has a unique solution $x$. 

We replace $x$ by the expression $v + w + 1$. It follows that $x$ is exactly the number of solutions of the non-negative exponential Diophantine equation equivalent with:
$$ (v + w + 1) \bmod (B^{m+1}) = c_0\,\,\wedge \,\,  N(m,n,P,v + w + 1,\max ,t) = t.$$

We consider the counting solution term $M(m,n,P, c_0, \max,t)$ and we observe that:
$$M(m,n,P,c_0, \max ,t) = x.$$
It follows that the final configuration of the register machine is given by the arithmetic term:
$$C(m,n,P, c_0, \max,t) = \left \lfloor \frac{M(m,n,P,c_0, \max ,t)}{B^{n(t-1)}} \right \rfloor.$$
We make the convention that the program $P$ computes a function $f$ value in  $a$ is always read in the register $x_1$ which contained the value $a$ at the beginning. This means that the start configuration was $\iota_0 = (1, a, 0, \dots, 0)$, encoded by $c_0 = 1 + aB$, where $B = m + a + t$.  Also, suppose that the computation time for $f(a)$ is bounded by a function $\theta(a)$, which is also an arithmetic term. In this case, we obtain the following arithmetic term representation of the function $f$:
$$f(a) = \left \lfloor \frac{C(m,n,P,1 + aB, a,\theta(a)) \bmod B^2}{B} \right \rfloor. $$
In general, to compute some function of arity $r$, $B = m + \max(1, x_1, \dots, x_r) + \theta(x_1, \dots, x_r)$, $t = \theta(x_1, \dots, x_r)$, $c_0(x_1, \dots, x_r) = 1 + x_1B + \dots + x_rB^r$ but the result saved in the register $x_1$ will be read in the same way. 

\section{A wise arithmetic term}

Fix once and for all an encoding of first-order formulas, axiom schemes, and proofs that includes the symbols $;$ and $\vdash$. Statements of the form $T \vdash \varphi$ (where $T$ is a finite list of axioms and schemes separated by $;$ and $\varphi$ is a sentence) and fully formalized proofs in $T$ are both represented by natural numbers. We may impose any convenient syntactic discipline on proofs (numbered lines, explicit derivation hints such as ``modus ponens of line 15 and line 8'', separation by $;$, etc.).

In this section we exploit the existence of a special Turing machine, as follows:

\begin{lemma}\label{lemmawiseturingmachine}
There is a Turing machine fulfilling the following job: The machine starts on some input consisting of two words. The first word is the code of a statement $T\vdash \varphi$. The second word encodes a natural number $k$. If there exists a proof of $\varphi$ in $T$ consisting of at most $k$ characters, the machine stops and outputs this proof. If no such a proof does exist, the machine stops and outputs $0$. 
\end{lemma} 

\begin{proof} 
The machine deterministically generates all words of length $\leq k$ and deterministically checks whether any is a valid proof of $\varphi$ in $T$. It is known that every such verification takes deterministic polynomial time in $k$, see Immerman \cite{immerman1999descriptive}. Once a proof is found, the machine outputs the proof. Otherwise it outputs $0$. 
\end{proof} 

\begin{theorem}\label{theoremwiseterm}
There is an arithmetic term $W(n,k)$ with the following property. If $n$ is the code of a statement $T\vdash \varphi$ and there is a proof of $\varphi$ in $T$ whose code $\pi$ has length $\leq k$, then $W(n,k) = \pi$. In all other situations, $W(n,k) = 0$.  
\end{theorem} 

\begin{proof}
Recall the Turing complete term: 
$$M(c, d, P, x, l, n). $$
Let $P_0 \in \mathbb N$ be the program of the machine described in \cref{lemmawiseturingmachine} and $c_0, d_0 \in \mathbb N$ the values corresponding to this machine. There are some arithmetic terms as follows: $x(n, k)$ computes the input for the machine, $l(x(n,k))$ computes the length of this input, $T(n,k)$ computes a value bigger than the number of steps necessary for the machine to halt on this input. The arithmetic term: 
$$W(n,k) = \left \lfloor \frac{M(c_0, d_0, P_0, x(l,k), l(x(n,k)), T(n,k))}{2^{c_0(3T(n,k)+2)(T(n,k)+l(x(n,k)))}} \right \rfloor$$
computes the output of the machine. 
\end{proof} 

\begin{cor}
Let $T$ be any recursively axiomatizable first-order theory that contains a fragment sufficient to represent all Kalmar elementary functions (for example $\mathrm{I}\Delta_0 + \exp$, $\mathrm{PRA}$, $\mathrm{PA}$, $\mathrm{ZFC}$, or $\mathrm{NBG}$).  
Then for every sentence $\varphi$ that is provable in $T$, there exists some $k \in \mathbb N$ such that the arithmetic term $W(n,k)$ outputs (the code of) a proof of $\varphi$ in $T$, where $n$ is the code of the statement $T \vdash \varphi$. 

However, there is \textbf{no} Kalmar elementary function $f : \mathbb N \to \mathbb N$ such that
\[
\forall n \Bigl( T \vdash \varphi_n \ \Longrightarrow\  \text{minimal proof length of $\varphi_n$ in $T$} \leq f(n) \Bigr),
\]
where $\varphi_n$ denotes the sentence whose code is $n$.
\end{cor}

In summary, every theorem of $T$ can be recovered by a fixed arithmetic term, but the required witness $k$ is not in general elementary recursive in the positional code of the sentence. Thus proof search for any recursively axiomatizable theory is internalized inside one fixed arithmetic term.

\begin{remark}
The same term yields a G\"odel-style self-referential sentence. Fix a formal theory $T$ (for example $\mathrm{ZFC}$) that is strong enough to formalize arithmetic syntax and proofs.

Let $\psi(m)$ be the formula
\[
\forall k \in \mathbb{N}, \; W(m,k)=0,
\]
where $m$ is a free variable intended to hold the positional code of a sentence. Let $q$ be the positional code of $\psi(m)$.

Let $\mathrm{sub}(p,n)$ be an elementary arithmetic term that substitutes the numeral for $n$ into the free variable of the formula coded by $p$ (built from the same pairing and projection primitives used throughout this paper). Define
\[
\varphi := \psi(\mathrm{sub}(q,q)).
\]
By construction, $\mathrm{sub}(q,q)$ is exactly the positional code of $\varphi$, so $\varphi$ asserts:
\[
\forall k \in \mathbb{N}, \; W(\mathrm{sub}(q,q),k)=0.
\]
Hence, if $T$ is consistent, then $T\nvdash \varphi$; moreover, in the standard model, $\varphi$ is true. Indeed, if $T\vdash \varphi$, then for some $k$ the term $W(\mathrm{sub}(q,q),k)$ outputs a nonzero proof code, contradicting the very statement of $\varphi$.
\end{remark}

Further consequences are immediate. The auxiliary term
$$I(n,k) = 1 \dotdiv (1 \dotdiv W(n,k))$$
outputs $1$ if there exists a $T$-proof of $\varphi$ whose code is at most $k$, and outputs $0$ otherwise.

Suppose now that every valid proof has code $> 2$. There is a term $\omega(n)$ which maps the code of $T\vdash \varphi$ to the code of $T\vdash \neg \varphi$. Define
$$J(n,k) := (1 \dotdiv (1 \dotdiv W(n,k))) + (2 \dotdiv (2 \dotdiv W(\omega(n),k))).$$
Then $J(n,k)=1$ if a bounded proof of $\varphi$ exists, $J(n,k)=2$ if a bounded proof of $\neg\varphi$ exists, and $J(n,k)=0$ otherwise. Therefore $\varphi$ is independent of $T$ if and only if
$$\forall k \in \mathbb N, \, J(n,k)=0.$$
In the non-independent case, the term
$$S(n,k) = W(n,k) + W(\omega(n),k)$$
returns a proof of $\varphi$ or of $\neg\varphi$. Assuming that the conclusion is encoded at the end of the proof string, a sufficiently large $k$ reveals which of the two is provable. In particular, $T$ is complete if and only if for every sentence $\varphi$ in its language there exists $k \in \mathbb N$ such that $S(n,k)\neq 0$, where $n$ is the code of $T\vdash \varphi$.

Let $\mathrm{ZFC}$ denote Zermelo--Fraenkel set theory with Choice, and let $Z$ be the code of the statement $\mathrm{ZFC}\vdash (0=0)\wedge(0\neq 0)$. Then $\mathrm{ZFC}$ is consistent if and only if
$$\forall k \in \mathbb N, \, W(Z,k)=0.$$
\begin{cor}
We showed in \cref{Sect.Introduction} that the identity problem for arithmetic terms is undecidable. Hence, if $\mathrm{ZFC}$ is consistent, then the identity problem for arithmetic terms has true instances that are not provable in $\mathrm{ZFC}$.
\end{cor}

Let $\mathrm{RH}$ denote the Riemann Hypothesis, and let $R$ be the code of $\mathrm{ZFC}\vdash \mathrm{RH}$. If for some $k \in \mathbb N$ one has
$$W(R,k)\neq 0,$$
then
$$(\mathrm{ZFC}\textrm{ is inconsistent}) \,\vee\, (\mathrm{RH}\textrm{ is true}).$$
For formal convenience, one may replace $\mathrm{ZFC}$ by the equivalent theory $\mathrm{NBG}$, which has a finite axiomatization.

Finally, consider decidable theories for which true sentences of length $l$ always admit proofs of length at most $L(l)$, with $L$ Kalmar elementary. For example, the first-order theory of the ordered field of real numbers has a decision procedure by quantifier elimination \cite{tarski1951decision}, and modern bounds are doubly exponential in the input length. Hence proof lengths are bounded by a corresponding arithmetic term $T(l)$. There is also a term $\xi(n)$ such that if $n$ is the code of a sentence $\varphi$ in the language of ordered fields, then $\xi(n)$ is the code of $T_{\mathbb R}\vdash \varphi$. Let $l(n)$ be the term computing input length, and define
$$K(n) = W(\xi(n), T(l(n))).$$
\begin{cor}
There exists an arithmetic term $K(n)$ such that if $n$ is the code of a sentence $\varphi$ true in the ordered field of real numbers, then $K(n)$ is the code of a proof of $\varphi$. As usual, evaluating this term uses a fixed number of arithmetic operations independent of the input. For other values of $n$, one has $K(n)=0$.
\end{cor}


\begin{thebibliography}{15}
\providecommand{\natexlab}[1]{#1}
\providecommand{\url}[1]{\texttt{#1}}
\expandafter\ifx\csname urlstyle\endcsname\relax
  \providecommand{\doi}[1]{doi: #1}\else
  \providecommand{\doi}{doi: \begingroup \urlstyle{rm}\Url}\fi

\bibitem{AroraBarak}{ Sanjeev Arora, Boaz Barak}. 
\newblock{Computational Complexity: A Modern Approach.}
\newblock{Cambridge University Press.}
\newblock{2009.}

\bibitem{Bruin} H. Bruin. 
\newblock{For almost every tent map, the turning point is typical.}
\newblock{\it Fundamenta Mathematicae.}
\newblock{155.3,}
\newblock{215 - 235,}
\newblock{1998.}

\bibitem{Collatz}{L. Collatz}.
\newblock{Aufgaben E}.
\newblock{\it Mathematische Semesterberichte}.
\newblock{1,}
\newblock{35,}
\newblock{1950.}

\bibitem[{P. Clote}(1964)]{Clote1995}
{P. Clote}.
\newblock Computation Models and Function Algebras,
\newblock in D. Leivant, ed.,
\newblock \textit{Logic and Computational Complexity. LCC 1994. Lecture Notes in Computer Science, vol 960.},
\newblock Springer,
\newblock Berlin, Heidelberg,
\newblock 1995.
\newblock URL \url{https://doi.org/10.1007/3-540-60178-3_81}

\bibitem[Grzegorczyk(1953)]{grzegorczyk1953someclasses}
A.~Grzegorczyk.
\newblock {Some Classes of Recursive Functions}.
\newblock \emph{Rozprawy Matematyczne}, \textbf{4}, 1953.
\newblock URL \url{http://matwbn.icm.edu.pl/ksiazki/rm/rm04/rm0401.pdf}.

\bibitem[Immerman(1999)]{immerman1999descriptive}
N.~Immerman.
\newblock \emph{Descriptive Complexity}.
\newblock Springer, 1999.
\newblock URL \url{https://doi.org/10.1007/978-1-4612-0539-5}.

\bibitem[Jones and Matiyasevich(1982)]{jones1982symmetric}
J.~P. Jones and Yu.~V. Matiyasevich.
\newblock {A New Representation for the Symmetric Binomial Coefficient and Its
  Applications}.
\newblock \emph{{Annales des Sciences Math\'ematiques du Qu\'ebec}}, 6\penalty0
  (1):\penalty0 81--97, 1982.

  Gaston Julia (1918): Sur l'itération des fonctions rationnelles, Journal de Mathématiques Pures et Appliquées 

\bibitem{Julia} Gaston Julia.
\newblock{Sur l'itération des fonctions rationnelles.}
\newblock{\it Journal de Mathématiques Pures et Appliquées.}
\newblock{8e série,}
\newblock{tome 1,}
\newblock{p. 47-245,}
\newblock{1918.}

\bibitem[Marchenkov(1980)]{marchenkov1980superposition}
S.~S. Marchenkov.
\newblock {A Superposition Basis in the Class of Kalmar Elementary Functions}.
\newblock \emph{\textit{Mathematical Notes of the Academy of Sciences of the
  USSR}}, \textbf{27}\penalty0 (3):\penalty0 161--166, 1980.
\newblock ISSN 0001-4346.
\newblock URL \url{https://doi.org/10.1007/BF01140159}.

\bibitem[Marchenkov(2007)]{marchenkov2007superposition}
S.~S. Marchenkov.
\newblock {Superpositions of Elementary Arithmetic Functions}.
\newblock \emph{\textit{Journal of Applied and Industrial Mathematics}},
\textbf{1}\penalty0 (3):\penalty0 351--360, 2007.
\newblock ISSN 1990-4789.
\newblock URL \url{https://doi.org/10.1134/S1990478907030106}. 

\bibitem{matiyasevich1993hilbert} { Y. Matiyasevich}. 
\newblock{ Hilbert's Tenth Problem.} 
\newblock MIT press, 
\newblock ISBN 0-262-13295-8, 
\newblock 1993.


\bibitem{May} R. M. May. 
\newblock Simple mathematical models with very complicated dynamics.
\newblock {\it Nature,}
\newblock 261(5560),
\newblock 459-467.

\bibitem{mazzanti2002plainbases}
S.~Mazzanti.
\newblock Plain bases for classes of primitive recursive functions.
\newblock \emph{\textit{Mathematical Logic Quarterly}}, \textbf{48}\penalty0
  (1):\penalty0 93--104, 2002.
\newblock ISSN 0942-5616.
\newblock URL
\url{https://doi.org/10.1002/1521-3870(200201)48:1%3C93::AID-MALQ93%3E3.0.CO;2-8}.

\bibitem[Mendelson(2015)]{mendelson2015introduction}
E.~Mendelson.
\newblock \emph{{Introduction to Mathematical Logic}}.
\newblock Chapman and Hall/CRC, New York, sixth edition, 2015.
\newblock URL \url{https://doi.org/10.1007/978-1-4615-7288-6}

\bibitem[{I. Oitavem}(1997)]{oitavem1997}
{I. Oitavem}.
\newblock {New Recursive Characterizations of the Elementary Functions and the
  Functions Computable in Polynomial Space}.
\newblock \emph{Revista Matemática de la Universidad Complutense de Madrid},
  10\penalty0 (1):\penalty0 109--125, 1997.
\newblock URL \url{http://eudml.org/doc/44242}.

\bibitem[Prunescu(2025)]{prunescu2025fibonacci2}
M.~Prunescu.
\newblock On other two representations of the $ \textrm{C} $-recursive integer
  sequences by terms in modular arithmetic.
\newblock \emph{Journal of Symbolic Computation}, 130, 2025.
\newblock URL \url{https://doi.org/10.1016/j.jsc.2025.102433}.

\bibitem[Prunescu and Sauras-Altuzarra(2024)]{prunescusauras2024factorial}
M.~Prunescu and L.~Sauras-Altuzarra.
\newblock An arithmetic term for the factorial function.
\newblock \emph{Examples and Counterexamples}, 5, 2024.
\newblock ISSN 2666-657X.
\newblock URL \url{https://doi.org/10.1016/j.exco.2024.100136}.

\bibitem{prunescusauras2025manyfunctions}
M.~Prunescu and L.~Sauras-Altuzarra.
\newblock Computational considerations on the representation of number-theoretic functions by arithmetic terms.
\newblock \emph{Journal of Logic and Computation}, 35, 2025.
\newblock URL \url{https://doi.org/10.1093/logcom/exaf012}.

\bibitem[Prunescu and Sauras-Altuzarra(2025)]{prunescusauras2025fibonacci}
M.~Prunescu and L.~Sauras-Altuzarra.
\newblock On the representation of C-recursive integer sequences
  by arithmetic terms.
\newblock \emph{Journal of Difference Equations and Applications}, 31, 2025.
\newblock URL \url{https://doi.org/10.1080/10236198.2025.2530478} 

\bibitem[Prunescu, Sauras-Altuzarra, Shunia(2025)]{prunescusaurasshuniaminimal} 
M.~Prunescu, L.~Sauras-Altuzarra and J.~M. Shunia.
\newblock A Minimal Substitution Basis for the Kalmar Elementary Functions
\newblock URL \url{https://arxiv.org/abs/2505.23787} 


\bibitem[Prunescu and Shunia(2024{\natexlab{a}})]{prunescushunia2024gcd}
M.~Prunescu and J.~M. Shunia.
\newblock Arithmetic-term representations for the greatest common divisor, 2024{\natexlab{a}}.
\newblock URL \url{https://arxiv.org/abs/2411.06430}.
\newblock Preprint.

\bibitem[Prunescu and Shunia(2024{\natexlab{b}})]{prunescushunia2024primes}
M.~Prunescu and J.~M. Shunia.
\newblock On arithmetic terms expressing the prime-counting function and the
  n-th prime, 2024{\natexlab{b}}.
\newblock URL \url{https://arxiv.org/abs/2412.14594}.
\newblock Preprint.

\bibitem[Prunescu and Shunia(2025)]{prunescushunia2025fibonacci3}
\newblock M. Prunescu and J. M. Shunia,
\newblock On Modular Representations of C-Recursive Integer Sequences,
\newblock \emph{Journal of Integer Sequences}, 28, 2025.
\newblock URL \url{https://cs.uwaterloo.ca/journals/JIS/VOL28/Prunescu/prunescu3.pdf}

\bibitem[{D. Rödding}(1964)]{Rodding1964}
{D. Rödding}.
\newblock {Über die Eliminierbarkeit von Definitions-schemata in der Theorie der rekursiven Funktionen}.
\newblock \emph{Mathematical Logic Quarterly},
  10, 315--330, 1964.
\newblock URL \url{https://doi.org/10.1002/malq.19640101806}.

\bibitem{Schoning2008}
\newblock Uwe Sch\"oning.
\newblock {Theoretische Informatik - kurz gefa\ss t.} 
\newblock Spektrum Akademischer Verlag, 
\newblock Heidelberg,
\newblock  ISBN 978-3-8274-1824-1,
\newblock 2008.

\bibitem{shuniasauras2025}
\newblock J. M. Shunia and L. Sauras Altuzarra,
\newblock Arithmetic Terms for Sums of Multinomial Coefficients,
\newblock \textit{The Ramanujan Journal}.
\newblock Volume 68.
\newblock Issue 3.
\newblock 2025.
\newblock \url{https://doi.org/10.1007/s11139-025-01222-3} 

\bibitem[Tarski(1951)]{tarski1951decision}
A.~Tarski.
\newblock \emph{A Decision Method for Elementary Algebra and Geometry}.
\newblock University of California Press, Berkeley, 1951.

Turing, A. M. (1936). "On Computable Numbers, with an Application to the Entscheidungsproblem" (PDF). Proceedings of the London Mathematical Society. 2. 42 (published 1937): 230–265.

\bibitem{Turing} A. M. Turing.
\newblock On Computable Numbers, with an Application to the Entscheidungsproblem.
\newblock {\it Proceedings of the London Mathematical Society.}
\newblock 2.42.
\newblock 1936.

\end{thebibliography}
\end{document}